\documentclass[11pt,a4paper,oneside,reqno]{amsart}
\usepackage[utf8]{inputenc}
\usepackage[style=alphabetic, maxbibnames=9]{biblatex}
\usepackage[T1]{fontenc}

\usepackage{setspace}
\usepackage[totalwidth=15cm,totalheight=22cm]{geometry}

\usepackage[dvips]{graphicx}
\usepackage{epsfig}
\usepackage{amsmath,amssymb,amscd,amsthm}
\usepackage{subfigure}

\usepackage{latexsym}
\usepackage{graphics}

\usepackage{colortbl} 
\usepackage{color}
\usepackage{ae}
\usepackage{enumitem}

\usepackage[all]{xy}
\usepackage{scalerel}
\usepackage{tikz}
\usepackage{cancel}
\usepackage{bbm,amsbsy}
\usepackage{mathrsfs}  
\usepackage[colorlinks=true]{hyperref}
\hypersetup{
	colorlinks,
	citecolor=green,
	filecolor=black,
	linkcolor=black,
	urlcolor=black
}
\usepackage[normalem]{ulem}
\usepackage{nicefrac}
\usepackage{xfrac}
\usepackage{faktor}
\usepackage{tikz-cd} 

\usepackage{nomencl}
\usepackage[normalem]{ulem}

\makenomenclature

\makeatletter
\newtheorem*{rep@theorem}{\rep@title}
\newcommand{\newreptheorem}[2]{
	\newenvironment{rep#1}[1]{
		\def\rep@title{#2 \ref{##1}}
		\begin{rep@theorem}}
		{\end{rep@theorem}}}
\makeatother

\makeatletter
\newtheorem*{rep@cor}{\rep@title}
\newcommand{\newrepcor}[2]{
	\newenvironment{rep#1}[1]{
		\def\rep@title{#2 \ref{##1}}
		\begin{rep@cor}}
		{\end{rep@cor}}}
\makeatother

\makeatletter
\newtheorem*{rep@prop}{\rep@title}
\newcommand{\newrepprop}[2]{
	\newenvironment{rep#1}[1]{
		\def\rep@title{#2 \ref{##1}}
		\begin{rep@prop}}
		{\end{rep@prop}}}
\makeatother

\theoremstyle{plain}
\newtheorem{cor}{Corollary}[section]

\newtheorem{theorem}[cor]{Theorem}
\newtheorem*{theorem*}{Theorem}

\newtheorem{prop}[cor]{Proposition}
\newtheorem*{prop*}{Proposition}

\newtheorem{alphatheorem}{Theorem}

\newrepcor{cor}{Corollary}
\newreptheorem{theorem}{Theorem}
\newrepprop{prop}{Proposition}

\newtheorem{lemma}[cor]{Lemma}

\theoremstyle{definition}
\newtheorem{defi}[cor]{Definition}
\newtheorem{remark}[cor]{Remark}
\newtheorem*{remark*}{Remark}

\theoremstyle{remark}

\newtheorem*{notation*}{Notation}

\newlist{steps}{enumerate}{1}
\setlist[steps, 1]{itemsep=8pt,leftmargin=0cm,itemindent=.5cm,labelwidth=\itemindent,labelsep=0cm,align=left,label = \textbf{\emph{Step \arabic*}:\,}}

\newcommand{\C}{{\mathbb C}}
\newcommand{\R}{{\mathbb R}}

\newcommand{\Hyp}{\mathbb{H}}

\newcommand{\SL}{\mathrm{SL}}
\newcommand{\PSL}{\mathrm{PSL}}

\newcommand{\GG}{\mathbb{G}}

\newcommand{\Isom}{\mathrm{Isom}}

\newcommand{\inner}[1] {\langle #1 \rangle}
\newcommand{\inners}{\langle \cdot, \cdot \rangle}

\newcommand{\QF}{\mathrm{QF}}
\newcommand{\TSTS}{\mathcal T(S)\times \mathcal T(\overline{S})}
\newcommand{\CSCS}{\mathcal C(S)\times \mathcal C(\overline{S})}
\newcommand{\ccpair}{(c_1, \overline{c_2})}
\newcommand{\ccpairclass}{([c_1], [\overline{c_2}])}
\newcommand{\CP}{\mathbb{CP}^1}
\newcommand{\phee}{\varphi}
\newcommand{\RR}{\mathrm R}

\newcommand{\HQD}{\mathrm{HQD}}

\usepackage{enumitem}

\newlist{primenumerate}{enumerate}{1}
\setlist[primenumerate,1]{label={(\arabic*$'$)}}

\begin{document}

	\setcounter{secnumdepth}{2}
	\setcounter{tocdepth}{2}
	
	\title[A metric uniformization model for $\QF(S)$]{A metric uniformization model for the Quasi-Fuchsian space}

\author[Christian El Emam]{Christian El Emam}
\address{Christian El Emam: University of Torino, Dipartimento di Matematica ``Giuseppe Peano", Via Carlo Alberto, 10, 10123 Torino, Italy.} \email{christian.elemam@unito.it}

	\maketitle
	
	\vspace{-0.6cm}

	\begin{abstract}
		We introduce and study a novel metric uniformization model for the quasi-Fuchsian space $\QF(S)$, defined through a class of $\C$-valued bilinear forms on $S$, called Bers metrics, which coincide with hyperbolic Riemannian metrics along the Fuchsian locus. 
		
		By employing this approach, we present a new model of the holomorphic tangent bundle of $\QF(S)$ that extends the metric model for the Teichmüller space defined by Berger and Ebin, and give an integral representation of the Goldman symplectic form and of the holomorphic extension of the Weil-Petersson metric to $\QF(S)$, with a new proof of its existence and non-degeneracy. We also determine new bounds for the Schwarzian of Bers projective structures extending Kraus' estimate.
		Lastly, we use this formalism to give alternative proofs to several classic results in quasi-Fuchsian theory.
	\end{abstract}

	\tableofcontents

\section{Introduction}

Let $S$ be a closed oriented surface of genus $\mathrm{g}\ge 2$. The quasi-Fuchsian space $\QF(S)$ is a distinctive neighborhood of the space of Fuchsian representations inside the $(\pi_1(S), \PSL(2,\C))$-character variety. Quasi-Fuchsian hyperbolic manifolds and the space $\QF(S)$ have been widely studied in the last century. One of the most powerful theorems in this area is Bers' Simultaneous Uniformization Theorem (see Theorem \ref{Bers Theorem}), showing that $\QF(S)$ is biholomorphic to $\mathcal T(S)\times \mathcal T(\overline S)$, where $\mathcal T(S)$ is the Teichmüller space of $S$. 

In this paper, we approach a new \emph{metric} perspective in the study of the geometry of quasi-Fuchsian space. In fact, Bers Theorem suggests the definition of a family of $\C$-valued bilinear forms on $S$, that we call \emph{Bers metrics} (see Sections \ref{introintro: metric model} and \ref{subsection hyperbolic complex metrics}), and a way to associate to each of them a quasi-Fuchsian holonomy, in a way that extends the holonomy assignment of hyperbolic Riemannian metrics to Fuchsian representations. One useful aspect of this approach is that there is a natural way to deform Bers metrics into new Bers metrics by adding holomorphic quadratic differentials, inducing a natural way to deform quasi-Fuchsian representations (see Theorem \ref{thm intro A}). 
This approach comes with several consequences:
\begin{itemize}[leftmargin=*, noitemsep,topsep=0pt]
	\item We give a new metric model for the holomorphic tangent bundle of $\QF(S)$, which extends the metric model for the tangent bundle of the space of hyperbolic metrics up to isotopy as in \cite{BergerEbin} and \cite{FischerTromba} (see Remark \ref{rmk: metric model Teichmuller}).
	\item We give an integral description for the Goldman symplectic form on $\QF(S)$ and of the holomorphic extension of the Weil-Petersson metric to $\QF(S)$ (which coincides with the one showed in \cite{LoustauSanders}), and provide a new proof of its holomorphicity and non-degeneracy. 
	\item We present a new bound for the Schwarzian of Bers projective structures. More precisely, we give a lower bound for the distance from the boundary of any point inside the open bounded subset $\mathcal{S}_{[c]}^+\subset \HQD([c])$ of the Schwarzian derivatives of Bers projective structures for $[c]\in \mathcal T(S)$.
	\item  We give new proofs and interpretations of a few well-known results, about the differential of the Schwarzian map, the description of the affine structures on $\mathcal T(S)$ from Bers embeddings, and McMullen's quasi-Fuchsian reciprocity Theorem.
\end{itemize}

Many of the tools developed in this paper have been used in recent works \cite{ElSa,ElSa2} to prove results in higher Teichmüller theory.

\subsection{The metric uniformizing approach for quasi-Fuchsian space}
\label{introintro: metric model}

The complex Lie group $\PSL(2,\C)$ acts on $\Hyp^3$ by isometries and by Möbius maps on its visual boundary $\CP\cong \partial{\Hyp^3}$.

We say that a discrete and faithful representation $\rho\colon \pi_1(S)\to \PSL(2,\C)$ is \textbf{quasi-Fuchsian} if its limit set $\Lambda_\rho\subset \CP$ is a Jordan curve on $\CP$, hence its complementary is the disjoint union of two topological disks. The reader can find a short introduction to this topic in Section \ref{introsection QF}.

Let $\mathcal C(S)$ denote the space of complex structures on the oriented surface $S$. The most important theorem about quasi-Fuchsian representations is Bers' Simultaneous Uniformization Theorem (\cite{BersUniformization}), which we recall in the following version.

\begin{theorem} [Bers' Simultaneous Uniformization Theorem]
	\label{Bers Theorem}
	For all $(c_1, \overline{c_2})\in \CSCS$, there exist:
	\begin{itemize}[noitemsep,topsep=0pt]
		\item a unique quasi-Fuchsian representation $\rho\colon \pi_1(S)\to \PSL(2, \mathbb C)$ up to conjugacy, 
		\item a unique $\rho$-equivariant holomorphic diffeomorphism \[\sigma_1= \sigma_1(c_1,\overline {c_2})\colon (\widetilde S, c_1) \to \Omega_\rho^+\ ,\]
		\item  a unique $\rho$-equivariant antiholomorphic diffeomorphism \[\overline{\sigma_2}=\overline{\sigma_2}(c_1,\overline {c_2})\colon (\widetilde S, \overline{c_2}) \to \Omega_\rho^-\ .\]
	\end{itemize}
where $\Omega_\rho^+\sqcup \Omega_\rho^-=\CP\setminus \Lambda_\rho$.
	This correspondence defines a biholomorphism \[\mathfrak{B}\colon \TSTS\xrightarrow{\sim}\QF(S).\]
\end{theorem}

In the paper, we will mostly see $\mathfrak{B}$ as an identification.

Now, for each $c\in \mathcal C(S)$, the uniformization theorem provides a $(\pi_1(S),\PSL(2,\R))$-equivariant biholomorphism $\sigma\colon (\widetilde S, c)\to H\subset \C$ from its universal cover to the upper half plane, and the hyperbolic metric in the conformal class of $c$ is the pull-back through $f$ of the Riemannian metric on the hyperbolic space in its half-plane model, namely 
\[
 \sigma^*\left(\frac 1 {Im(z)^2} dzd\overline z \right)= -\frac 4 {(\sigma-\overline \sigma)^2}dfd\overline f\ 
\] 
which, being $\pi_1(S)$-invariant, defines a hyperbolic metric on $S$.

Bers Theorem suggests a way to extend this construction, assigning to each element of $\CSCS$, a symmetric 2-tensor on $S$ in the following way:

\begin{equation}
	\label{eq: da cscs a Bers metrics}
	\begin{split}
		g\colon	\CSCS & \to \Gamma (Sym_2(TS, \mathbb C))\\
		(c_1, \overline{c_2}) &\mapsto g(c_1, \overline {c_2}):= -\frac 4 {(\sigma_1-\overline{\sigma_2})^2} d\sigma_1\cdot d\overline{\sigma_2}
	\end{split}
\end{equation}
where $\sigma_1=\sigma_1(c_1,\overline{c_2})$ and $\overline{\sigma_2}=\overline{\sigma_2}(c_1,\overline{c_2})$ are as in the statement of Bers Theorem \ref{Bers Theorem}, and $\Gamma (Sym_2(TS, \mathbb C))$ denotes the space of smooth sections of $\mathbb C$-valued symmetric 2-forms on $TS$ (in other words, $Sym_2(TS, \mathbb C)= Sym_2(TS, \mathbb R) \oplus i Sym_2(TS, \mathbb R)$). One can easily show that the tensors defined this way are well-defined on $\widetilde S$, namely they depend neither on the choice of the affine chart of $\CP$ (containing the images of $\sigma_1$ and $\overline{\sigma_2}$), nor on the conjugacy class of the quasi-Fuchsian representation in the construction. In addition, it is clearly $\pi_1(S)$-invariant, defining a symmetric $\C$-valued tensor on $S$. Finally, if $c_1=c_2=:c$, then $g(c,\overline{c})$ is just the hyperbolic metric in the conformal class of $c$.

We call $\textbf{Bers metrics}$ the bilinear forms in the form $g(c_1,\overline{c_2})$ as in \eqref{eq: da cscs a Bers metrics}.

In fact, Bers metrics admit an intrinsic characterization: they correspond to the connected component of positive complex metrics of constant curvature $-1$ that contains hyperbolic Riemannian metrics (see Section \ref{subsection hyperbolic complex metrics} below).

In a sense, Bers metrics allow to \emph{uniformize} $\CSCS$ in a way that extends the classic uniformization of complex structures with hyperbolic Riemannian metrics.

We define the \textbf{holonomy} of $g=g(c_1, \overline{c_2})$ as $[g]=([c_1], [\overline{c_2}])\in \QF(S)$.

A useful aspect of this approach is that (pointwise) holomorphic deformations of Bers metrics induce holomorphic deformations of the holonomy in $\QF(S)$ (see Theorem \ref{thm: schwartzian}). In fact, the first part of the paper discusses an interesting type of holomorphic deformation, which is the following.
Let $\HQD(c)$ denote the space of holomorphic quadratic differentials for the Riemann surface $(S,c)$. Then, if $g=g(c_1,\overline{c_2})$ is a Bers metric, for all $q_1\in \HQD(c_1)$ and $\overline{q_2}\in \HQD(\overline{c_2})$ small enough,
\[
g+q_1 \qquad \qquad \text{and} \qquad \qquad g+\overline{q_2}
\] 
are Bers metrics as well. This comes with several interesting consequences, as explained in the following theorem (in the text, they follow from Propositon \ref{prop: g+q same curvature}, Proposition \ref{prop: star-shaped}, Theorem \ref{thm: schwartzian}, and Corollary \ref{cor: g+q isotopy independent})

\begin{alphatheorem}
	\label{thm intro A}
	Let $\ccpair\in \CSCS$ and consider the Bers metric $g=g(c_1, \overline{c_2})$.
	
	There exists an open subset $U_g\subset \HQD(c_1)$, $0\in U_g$, such that, for all $q_1\in U_g$, $g+q_1$ is a Bers metric.
	The map
			\begin{align*}
		U_g&\to \TSTS\\
		q_1&\mapsto [g+q_1]
	\end{align*}
	gives a holomorphic local parametrization of the Bers slice $\{[c_1]\}\times \mathcal T(\overline S)$, which is essentially independent from the choice of $c_1$ and $\overline{c_2}$ in their isotopy classes $[c_1]\in \mathcal T(S)$ and $[\overline{c_2}]\in \mathcal T(\overline S)$.
\end{alphatheorem}

An analog statement follows for deformations of the form $\overline{q_2}\mapsto g\ccpair+\overline{q_2}$.

In the paper, parts of these results are stated in greater generality for a wider class of complex metrics, called upper-projective and lower-projective complex metrics (see Section \ref{sec Bers metrics})

As a result, we get a \emph{uniformizing metric} model of the holomorphic tangent bundle to quasi-Fuchsian space given by 
\begin{equation}
	T_{([c_1],[\overline{c_2}] )} \QF(S)= \{q_1+\overline{q_2}\ |\  q_1\in \HQD(c_1),\  \overline{q_2}\in \HQD(\overline{c_2}) \}, 
\end{equation}

which is a vector subspace of the space of smooth sections of $Sym_2(TS, \mathbb C)$ and which extends the metric model for the tangent bundle of $\mathcal T(S)$ on the Fuchsian locus (see Remark \ref{rmk: metric model Teichmuller}).

\subsection{Metric deformation and Schwarzian deformation}

Bers Theorem associates to each pair $(c_1,\overline{c_2})\in \CSCS$ two projective structures $(\sigma_1(c_1, \overline{c_2}), \rho)$ and $(\overline{\sigma_2}(c_1, \overline{c_2}), \rho)$, for the complex structures $c_1$ and $\overline{c_2}$ respectively.

Consider the maps
\begin{equation}
	\label{eq: schw+ intro}
	\begin{split}
\mathbf{Schw}_+\colon	\{\text{Bers metrics}\} &\to \mathbf{HQD}(S)\\
	g(c_1, \overline{c_2}) &\mapsto Schw(\sigma_1(c_1, \overline{c_2}) )	
\end{split}
\end{equation}
and 
\begin{equation}
	\label{eq: schw- intro}
	\begin{split}
	\mathbf{Schw}_-\colon	\{\text{Bers metrics}\} &\to \mathbf{HQD}(\overline S)\\
	g(c_1, \overline{c_2}) &\mapsto Schw(\overline{\sigma_2}(c_1, \overline{c_2}))
\end{split}
\end{equation}
where $Schw$ denotes the Schwarzian derivative (see Section \ref{introsection proj structures}).

The following theorem (Theorem \ref{thm: schwartzian} in the paper) shows the relation between the metric deformation of $g$ with $g+q_1$ and $g+\overline{q_2}$ and the deformation of the Schwarzian of the corresponding Bers' projective structures.

\begin{alphatheorem}
	Let $g=g(c_1,\overline{c_2})$ be a Bers metric. 
	
	Then,
	\[
	\mathbf{Schw}_+(g+q_1)= \mathbf{Schw}_+(g) -\frac 1 2 q_1
	\]
	for all $q_1\in \HQD(c_1)$ such that $g+q_1$ is a Bers metric, and
	\[
	\mathbf{Schw}_-(g+\overline{q_2}) = \mathbf{Schw}_-(g) -\frac 1 2 \overline{q_2} 
	\]
	for all $\overline{q_2}\in \HQD(\overline{c_2})$ such that $g+\overline{q_2}$ is a Bers metric.
\end{alphatheorem}

\subsection{The holomorphic extension of Weil-Petersson metric}

A few extensions of the Weil-Petersson metric to $\QF(S)$ have been studied in literature.
We mention, for instance, that in \cite{bridgeman2010hausdorff} and \cite{bridgeman2015pressure}, the authors define and study a Riemannian extension. 
 
In \cite{LoustauSanders}, Loustau and Sanders show that the Weil-Petersson metric extends to a holomorphic Riemannian metric (see Section \ref{introsection: holo Riemannian metrics} for the definition) on the quasi-Fuchsian space, whose real part is therefore a pseudo-Riemannian metric on $\QF(S)$. We show that the metric formalism allows to give an alternative proof of the existence and uniqueness of the holomorphic extension of the Weil-Petersson metric, providing a new explicit description of it.

Let us define it. Let $(c_1, \overline{c_2})\in \CSCS$ and consider the corresponding Bers metric $g=g(c_1, \overline{c_2})$.
The metric $g$ induces a $\C$-bilinear form on the bundle $Sym_2(T S)\oplus i Sym_2(T S)$ given by
	\[
<\tau_1, \tau_2>_g= \sum_{i,j,k, \ell=1}^2 \tau_1(X_i, X_j)\tau_2(X_k,X_\ell) g^{ik}g^{j\ell} \ .
\]
where $\{X_1, X_2\}$ is any local basis for $TS$ and $(g^{ij})$ is the inverse matrix of the representative matrix $(g_{ij})$.

Bers metrics come with a natural notion of area form compatible with the orientation (see Remark \ref{rmk: area form}), namely
\[
dA_g= -\frac{2i}{(\sigma_1-\overline{\sigma_2})^2}d\sigma_1\wedge d\overline{\sigma_2}\ .
\]

One can therefore define a $\C$-bilinear form on $T_{([c_1],[\overline{c_2}])} \QF(S)= \{q_1+\overline{q_2}\ |\ q_1\in \HQD(c_1), \overline{q_2}\in \HQD(\overline{c_2}) \}$ by considering
	\begin{equation}
	\inner{\tau_1, \tau_2}_g= \frac 1 {8}\int_S <\tau_1, \tau_2>_g dA_g
\end{equation}

The following theorem follows from Proposition \ref{prop: generalita su metrica Riemanniana holo}, Theorem \ref{thm: holomorphic Riemannian}, Corollary \ref{cor: the metric is unique}, and Corollary \ref{cor: mapping class group invariance}.

\begin{alphatheorem}
	The bilinear form $\inners_g$ only depends on $g$ through its holonomy, and it defines a holomorphic Riemannian metric on $\QF(S)$ which coincides with the Weil-Petersson metric on the Fuchsian locus and which is invariant under the action of the mapping class group. 
\end{alphatheorem}
	By uniqueness of the holomorphic extension (see Corollary \ref{cor: the metric is unique}), this coincides with the holomorphic extension introduced in \cite{LoustauSanders}.

Explicitly, $\inners_g$ can be seen as follows:
\begin{enumerate}
	\item $\inner{q_1, q'_1}_g=0$, for all $q_1, q'_1\in \HQD(c_1)$.
	\item $\inner{\overline{q_2}, \overline{q_2}'}_g=0$ for all $\overline {q_2}, \overline{q_2}'\in \HQD(\overline{c_2})$\ .
	\item Let $g=\varrho dzd\overline w$, $q_1=\phee dz^2 \in \HQD(c_1)$, $\overline{q_2}=\overline{\psi} d\overline w^2\in \HQD(\overline{c_2})$, then
	\[
	\inner{q_1, \overline{q_2}}_g=\frac 1 4  i \int_S \frac{\phee\cdot \overline{\psi}}{\varrho}dz\wedge d\overline w	\]
\end{enumerate}

As a consequence we get an integral description of the Goldman symplectic form $\omega_G$ on $\QF(S)$ given by 
\[
\omega_G(q_1, \overline{q_2}) _{|([c_1], [\overline{c_2}])}= 2 i \int_S \frac{\phee\cdot \overline{\psi}}{\varrho}dz\wedge d\overline w	 
\] 
with Bers slices being Lagrangian (see Proposition \ref{prop: Goldman symp}).

\subsection{Bounds for the Schwarzian of Bers projective structures}

The maps $\mathbf{Schw}_+$ and $\mathbf{Schw}_-$ defined above in \eqref{eq: schw+ intro} and \eqref{eq: schw- intro} descend to maps from $\QF(S)$, namely 
\begin{equation*}
	\begin{split}	
		Schw_+\colon \QF(S) &\to  \HQD(S)\\
		([c_1], [\overline{c_2}])&\mapsto [Schw(\sigma_1(c_1,\overline{c_2}) )]
	\end{split}
\qquad \qquad
	\begin{split}	
	Schw_-\colon \QF(S) &\to  \HQD(\overline S)\\
	([c_1], [\overline{c_2}])&\mapsto [Schw(\overline{\sigma_2}(c_1, \overline{c_2}) )]
\end{split}\ 
\end{equation*}
where $[\mathbf{Schw}_+(g(c_1, \overline{c_2}))]= Schw_+([c_1], [\overline{c_2}])$ and $[\mathbf{Schw}_-(g(c_1, \overline{c_2}))]= Schw_-([c_1], [\overline{c_2}])$

Both maps are known to be biholomorphisms onto their images (see for instance \cite{dumasprojective}).
In particular, for all given $([c_1], [\overline{c_2}])\in \mathcal T(S)\times \mathcal T(\overline S)$, 
\begin{align*}
	\mathcal S^+_{[c_1]}= Schw_+\Big(\{[c_1]\} \times \mathcal T(\overline S)\Big)\subset \HQD([c_1])\\ \mathcal S^-_{[\overline{c_2}]}:=  Schw_-\Big( \mathcal T(S) \times \{[\overline{c_2}]\} \Big)\subset \HQD([\overline{c_2}])
\end{align*}
are open subsets. 

The shape of these open subsets is not clear in general. By the classic Kraus-Nehari theorem (\cite{kraus1932zusammenhang}, \cite{Nehari}), we know that $	\mathcal S^+_{[c_1]}$ (resp. $\mathcal S^-_{[\overline{c_2}]}$) contains the ball of radius $\frac 1 2$ and is contained in the ball of radius $\frac 3 2$ centered in zero inside $\HQD([c_1])$ (resp. $\HQD([\overline{c_2}]$) with respect to the $L^\infty$-norm given by the hyperbolic metric. By the works of Lempert \cite{Lempert} and Markovic \cite{Markovic}, $\mathcal T(S)$ cannot be biholomorphic to a convex domain in $\C^{6\mathrm g-6}$, hence $\mathcal S^+_{[c_1]}$ and $\mathcal S^-_{[\overline{c_2}]}$ are not convex.

The metric formalism for $T\QF(S)$ allows to give explicit lower bounds on the radius of balls inside $\mathcal S^+_{[c_1]}$ and $\mathcal S^-_{[\overline{c_2}]}$, as shown in the following Theorem (Corollary \ref{cor: ball bound for schw} in the text, with finer estimates in Lemma \ref{lemma: bound qF} and Theorem \ref{thm: bound qF}).

\begin{alphatheorem}
	\label{thm: bound intro}
	Let $\ccpair \in \CSCS$, let $g=g(c_1,\overline{c_2})$ denote the corresponding Bers metric, and let $g_0=g(c_1, \overline{c_1})$ be the Riemannian hyperbolic metric in the conformal class of $c_1$.
	Let 
\begin{equation}
	\label{eq: def R intro}
0< R:= \frac 1 2 \min_S\  \left(  \Big(1- \left|\frac{\partial_{\overline z} w}{\partial_z w} \right| \Big) \left| \frac{dA_g }{dA_{g_0}} \right| \ \right)
\end{equation}
where $z$ and $\overline w$ are any local coordinates for $c_1$ and $\overline{c_2}$ respectively and $dA_g$ and $dA_{g_0}$ are the area forms of $g$ and $g_0$. 
	Then, \[B_{\infty}\Big(Schw_+([c_1], [\overline{c_2}]), R\Big)\subset \mathcal S^+_{[c_1]}\, \]
where $B_{\infty}$ denotes a ball for the $L^{\infty}$-norm (with respect to the hyperbolic metric) on $\HQD([c_1])$.

\end{alphatheorem}
For the case $c_1={c_2}$ we get $R=\frac 1 2$, which coincides with Kraus' Theorem in \cite{kraus1932zusammenhang}.
	
 Observe that the definition of $R$ in $\eqref{eq: def R intro}$ depends on the choice of $c_1$ and $\overline{c_2}$ in their isotopy classes $[c_1]\in \mathcal T(S)$, $[\overline{c_2}]\in \mathcal T(\overline S)$: in fact it would be interesting to determine for each $c_1\in \mathcal C(S)$ what choice of $\overline{c_2}$ in its isotopy class $[\overline{c_2}]$ maximizes the estimate $R$ (also see Remark \ref{rmk: sharpness of R}).

\section*{Aknowledgements}
We thank the referee for their valuable suggestions, in particular for the alternative proof of Theorem \ref{thm: schwartzian}, which improved its clarity and smoothness.

We are grateful to Francesco Bonsante for some useful technical advice and to Andrea Tamburelli, Andrea Seppi, and Filippo Mazzoli for inspiring conversations on immersions into $\SL(2,\C)$ that indirectly influenced this work. We also thank Jean-marc Schlenker for his mini-course at the University of Luxembourg, which provided helpful background, and Nathaniel Sagman for his suggestions on the state of the art

\section*{Funding}
The author has been supported by the FNR OPEN grant CoSH (O20/14766753/CoSH).

\section{Essential background and notation}
\label{sec: essential background}

Throughout the whole paper, $S$ denotes a closed oriented surface of genus $\mathrm g\ge 2$. We will use $\overline S$ to denote the surface with opposite orientation.

\subsection*{Notation} Geometric structures (such as complex structures, Bers metrics, quadratic differentials, and tensors in general) on $S$ will often be identified with their $\pi_1(S)$-invariant lifts on the universal covering space $\widetilde S$.

\subsection{The Teichmüller space}
\label{introsection Teich space}

Let $\mathcal C(S)$ denote the \textbf{space of complex structures} on $S$ that are compatible with its orientation. 

Let $\mathrm{Diff}_+(S)$ denote the group of orientation-preserving diffeomorphisms of $S$, let $\mathrm{Diff}_0(S)$ denote the subgroup of  $\mathrm{Diff}_+(S)$ consisting of diffeomorphisms of $S$ that are isotopic to the identity, and define the \textbf{mapping class group} of $S$ as $MCG(S):=\nicefrac{\mathrm{Diff}_+(S)}{\mathrm{Diff}_0(S)}$. Elements in $\mathrm{Diff}_+(S)$ will be often identified with their $\pi_1(S)$-invariant lifts to the universal cover $\widetilde S$.

The \textbf{Teichmüller space} $\mathcal T(S)=\nicefrac{\mathcal C(S)}{\mathrm{Diff}_0(S)}$ is the space of complex structures on $S$ up to isotopy. We denote the equivalence class of $c$ with $[c]$. The mapping class group $MCG(S)$ acts naturally on $\mathcal T(S)$.

We say that a representation $\rho\colon \pi_1(S)\to \PSL(2,\R)$ is (orientation-preserving and) \textbf{Fuchsian} if it is discrete and faithful and its extension to the boundary $\partial \pi_1(S)\cong \partial \widetilde S\to \partial {\Hyp^2}$ is an orientation-preserving homeomorphism. The \textbf{Fuchsian space} $\mathrm{Fuch}(S)$ is the space of conjugacy classes of Fuchsian representations by elements in $\PSL(2,\R)$. 

The Uniformization Theorem provides a correspondence between $\mathcal T(S)$ and the \textbf{Fuchsian space} $\mathrm{Fuch}(S)$, obtained as follows.  $H\subset \C$ denoting the upper half-plane, for all $c\in \mathcal C(S)$ there exists a biholomorphism $(\widetilde S, c) \to H$ which is unique up to conjugation with elements of $\PSL(2,\R)$, and which is equivariant for a Fuchsian $\rho\colon \pi_1(S)\to \PSL(2,\R)$: the map sending $c$ to the conjugacy class of $\rho$ descends to a bijection between $\mathcal T(S)$ and $\mathrm{Fuch}(S)$ called \textbf{holonomy map} which assigns by pull-back a smooth structure on $\mathcal T(S)$.

By pulling-back the hyperbolic metric of $\Hyp^2$ in the upper-half plane model through the uniformizing map, we get a bijection between $\mathcal T(S)$ and the space $Met_{-1}(S)$ of hyperbolic metrics on $S$ up to isotopy, whose inverse is just the map sending a Riemannian metric to its conformal class.

\subsection{Beltrami coefficients}
\label{introsection Beltrami}

The content of Sections \ref{introsection Beltrami} and \ref{introsection T and T*} are well-known to experts. There are several surveys on these topics, such as \cite{Otalhandbook},\cite{ahlfors2006lectures},\cite{farb2011primer},\cite{hubbard2016teichmuller},\cite{petri2019teichmuller}, \cite{Morrey}, and many more.
We recall here some essential aspects of Beltrami equations and Beltrami differentials leading to the description of the tangent bundle to Teichmüller space.
\vspace{5pt}

Fix $c_0\in \mathcal C(S)$, let $\rho_0$ be its Fuchsian holonomy, and denote $\Gamma=\rho_0(\pi_1(S))\subset \PSL(2,\R)$.

We define the subspace  ${L^\infty}(\Gamma)\subset L^{\infty}(H,\C)$ given by Lebesgue-integrable functions $\mu\colon H\to \mathbb C$ such that $\mu(\gamma (z))= \frac{\overline{\gamma'(z)}}{\gamma'(z)} \mu(z)$ for all $\gamma\in \Gamma$: in other words, the section $\mu \frac{d\overline z}{dz}$ of $\mathbb CT^*S\otimes \mathbb CT^{**}S$ (or equivalently $\mu\ {d\overline z}\otimes\partial_z$ in the identification $\mathbb CT^{**}S\cong \mathbb CTS$) is $\Gamma$-invariant on $\widetilde S$. Also denote $L^\infty_1(\Gamma)$ as the intersection of ${L^\infty}(\Gamma)$ with the unit ball of $L^{\infty}(H,\C)$. 
Observe that $L^\infty(\Gamma)$ can be seen as the tangent space in zero of $L^\infty_1(\Gamma)$: we call $\Gamma$-\textbf{Beltrami coefficients} (or just Beltrami coefficients) elements in $L^\infty_1 (\Gamma)$ and denote them with $\mu$, while we call $\Gamma$-\textbf{Beltrami differentials} (or just Beltrami differentials) elements in $L^\infty(\Gamma)$ and denote them with $\beta=\mu\frac{d\overline z}{dz}$.

By the classic theory of Beltrami equations (e.g. see \cite{Morrey}, \cite{Otalhandbook}), for all $\mu$ such that $\|\mu\|_{\infty}<1$, there exists a unique quasi-conformal map $f^{\mu}\colon  \C\to  \C$ such that: 
\begin{itemize}[noitemsep,topsep=0pt]
	\item $\frac{\partial f^{\mu}}{\partial \overline z}= \mu \frac{\partial f^{\mu}}{\partial z}$ almost everywhere on $H$,
	\item the restriction of $f^{\mu}$ to $(-H)$ is holomorphic,
	\item $f^{\mu}(0)=0$ and $f^{\mu}(1)=1$ 
\end{itemize}
Moreover, $f^{\mu}$ is equivariant for a discrete and faithful representation $\rho\circ \rho_0^{-1}\colon \Gamma\to \PSL(2,\mathbb C)$, with $\rho\colon \pi_1(S)\to \PSL(2,\mathbb C)$ acting freely and properly discontinuously on $f^{\mu}(H)$. Hence $f^{\mu}$ induces a quasi-conformal map between the Riemann surfaces $f^\mu\colon\nicefrac{H}{\Gamma}\to \nicefrac{f^{\mu}(H)}{(\rho\circ \rho_0^{-1})(\Gamma)}$. 

Define $\mathrm{Belt}_1(\rho_0)$ as the quotient of $L^\infty_1(\Gamma)$ defined by the relation $\mu\sim \mu'$ equivalently if (with reference to the construction above):
\begin{itemize}[noitemsep,topsep=0pt]
	\item $f^\mu$ and $f^{\mu'}$ are equivariant for the same representation $\rho\colon \pi_1(S)\to \PSL(2,\C)$;
	\item there exists an element $\phee\in \mathrm{Diff}_0(S)$ such that $f^{\mu'}=f^{\mu}\circ \phee$;
	\item the pull-back complex structures on $S$ defined by $f^\mu$ and $f^{\mu'}$ are isotopic.
\end{itemize} 

The last identification provides a map $\mathrm{Belt}_1(\rho_0) \xrightarrow{\sim} \mathcal T(S)$ which is in fact a diffeomorphism.

Under this identification, the tangent space $T_{[c_0]}\mathcal T(S)$ can be seen as the quotient of $L^\infty(\Gamma)$ with the kernel of the differential map of the projection $L^\infty_1(\Gamma)\to \mathrm{Belt}_1(\rho_0)$, as we describe in the next section.

\subsection{Tangent and cotangent bundle to Teichmüller space}
\label{introsection T and T*}

Let $c\in \mathcal C(S)$. A \textbf{holomorphic quadratic differential} on $(S,c)$ is a holomorphic section of the square of the holomorphic cotangent bundle. We will denote the complex vector space of holomorphic quadratic differentials for $c$ with $\HQD(c)$. One can see (the lift to the universal cover of) a holomorphic quadratic differential on $S$ as a $\pi_1(S)$-invariant tensor $\phee dz^2$ on $\widetilde S$, where $z$ is a global holomorphic coordinate for (the lift of) $c$ on $\widetilde S$, and $\phee\colon \widetilde S\to \mathbb C$ is $c$-holomorphic.
We also denote $\HQD([c])=\faktor{\bigcup_{c\in [c]} \HQD(c)}{\mathrm{Diff}_0(S)}$, which is naturally a complex vector space with the natural identification $\HQD(c)\cong\HQD([c])$. By Riemann-Roch Theorem, $\HQD([c])$ has complex dimension $3\mathrm g-3$.

Recall the notation of the previous paragraph. The identification $\mathrm{Belt}_1(\rho_0)\xrightarrow{\sim} \mathcal T(S)$ defined above actually allows to give a description of the tangent and cotangent space to $\mathcal T(S)$ in $[c_0]$ as follows. 

Let $\rho_0$ be the Fuchsian representation corresponding to $c_0\in \mathcal C(S)$, and regard ${L^\infty}(\rho_0(\pi_1(S)))$ as the tangent space to $L^\infty_1(\rho_0(\pi_1(S)))$ in zero.

For all $q=\phee dz^2\in \HQD(c_0)$, $\beta=\mu \frac{d\overline z}{dz}\in {L^\infty}(\rho_0(\pi_1(S)))$, the 2-form $\phee \mu dz\wedge d\overline z$ on $\widetilde S$ is $\pi_1(S)$-invariant, so one can consider the pairing defined by
\begin{equation}
	\label{eq: pairing WP}
q(\beta):=\int_S \phee \cdot \mu \ \frac i 2 dz\wedge d\overline z \ .
\end{equation}
The vector subspace \[
\mathcal N=\left \{ \beta\in {L^\infty}(\rho_0(\pi_1(S))) \ |\ q(\beta)=0 \text{ for all } q\in \HQD(c_0) \right \}.
\] 
is precisely the kernel of the differential of the quotient map ${L^\infty_1}(\rho_0(\pi_1(S)))\to \mathrm{Belt}_1(\rho_0)$ in zero.

As a result, the tangent space in $[c_0]$ to Teichmüller space, canonically identified with $T_{[0]}\mathrm{Belt}_1(\rho_0)$, can be seen as the finite vector space
\[
T_{[c_0]}\mathcal T(S)\cong \mathrm{Belt}([c_0]):= \faktor{{L^\infty}(\rho_0(\pi_1(S)))}{\mathcal N}\ .
\]

Finally, denote by $h_0$ the hyperbolic metric in the conformal class of $c_0$. The map
\begin{align*}
	\HQD(\overline{c_0}) &\to \mathrm{Belt}([c_0])\\
	\overline q &\mapsto \left[ \frac {\overline q} {h_0}\right]= \left[ \frac{\overline \phee d\overline z}{\varrho dz} \right]
\end{align*}
where $h_0=\varrho dz d\overline z$, is a linear isomorphism \cite{Otalhandbook}, and it gives a nice set of representatives for the elements of $\mathrm{Belt}([c_0])$. The Beltrami differentials of the form $\frac {\overline q} {h_0}$ as above are called \textbf{harmonic Beltrami differentials}.

Since $q(\frac{\overline q}{h_0})>0$ for all $q\in \HQD(c_0)$, we get that the pairing on $\HQD([c_0])\times\mathrm{Belt}([c_0])$ given by $q([\beta]):=q(\beta)$ as in Equation \eqref{eq: pairing WP} is well-defined and non-degenerate, thus we get the identification $T^*_{[c_0] }\mathcal T(S)\cong\HQD([c_0])$.

Both $\HQD([c_0])$ and $\mathrm{Belt}([c_0])$ are naturally complex vector spaces, providing almost-complex structures on $T \mathcal T(S)$ and consistently on $T^*\mathcal T(S)$. As we will see in Section \ref{introsection QF}, this almost-complex structure is integrable, and $\mathcal T(S)$ has a natural structure of a complex manifold.

We will denote by $\textbf{HQD}(S)$ the space of quadratic differentials which are holomorphic with respect to some complex structure on $S$ compatible with the orientation (this is an infinite-dimensional Banach manifold) and with $\HQD(S)= \faktor{\textbf {HQD}(S)}{\mathrm{Diff}_0(S)}$ the vector bundle of rank $6\mathrm g-6$ on $\mathcal T(S)$ which is isomorphic to its cotangent bundle.

\subsection{Complex projective structures}
\label{introsection proj structures}

The complex Lie group $\PSL(2,\C)$ acts by isometries on the $3$-dimensional hyperbolic space $\Hyp^3$ and on its visual boundary at infinity $\CP$ by Möbius maps. 

A \textbf{(complex) projective structure} on $S$ is a $(\PSL(2,\C),\CP)$-structure, namely a maximal atlas of charts to open subsets of $\CP$ with changes of coordinates being restrictions of Möbius maps. By the classic theory of $(G,X)$-structures, this is equivalent to giving a \emph{developing map} $dev\colon \widetilde S\to \CP$ that is equivariant for some representation $\rho\colon \pi_1(S)\to \PSL(2,\C)$, called the \emph{holonomy} of the projective structure. A nice survey on complex projective structures is \cite{dumasprojective}. 

The space of projective structures up to isotopy, that we will denote by $\mathcal P(S)$, is a complex manifold, and the forgetful map $\mathcal P(S)\to \mathcal T(S)$ that maps a projective structure to the complex structure induced by the atlas is smooth (in fact holomorphic as we see in the next section).

Let $\mathcal X(\pi_1(S), \PSL(2,\C))$ denote the $\PSL(2,\C)$-\textbf{character variety}, which is the GIT quotient (or, from the topological viewpoint, the greatest Hausdorff quotient) of the space of conjugacy classes of representations in $\PSL(2,\C)$, namely $\faktor{\{\rho\colon \pi_1(S)\to \PSL(2,\C)\}}{\PSL(2,\C)}$. The character variety $\mathcal X(\pi_1(S), \PSL(2,\C))$ is naturally a complex algebraic variety, and the holonomy map
\begin{equation}
	\label{eq: holonomy of projective structures is a local biholo}
	\begin{split}
hol\colon	\mathcal P(S)&\to \mathcal X(\pi_1(S), \PSL(2,\C))\\
	[(dev,\rho)]&\mapsto [\rho]
	\end{split}
\end{equation}
is a local biholomorphism.

A projective structure on $S$ can be encoded in a tensor on $S$ by taking the Schwarzian derivative of the developing map, let us define it. Given an open subset $\Omega\subset \C$ and a local biholomorphism $f\colon \Omega\to\C$, we denote its \textbf{Schwarzian derivative} as the holomorphic quadratic differential 
\[
Schw(f):= \left( \left(\frac{f''}{f'}\right)'-\frac 1 2 \left(\frac{f''}{f'} \right)^2  \right) dz^2\ .
\] 
The key properties of the Schwarzian derivative are that $Schw(f)=0$ if and only if $f$ is Möbius, and the chain rule $Schw(f\circ g)=g^*(Schw(f)) +Schw(g)$. 

Given a projective structure $(f,\rho)$ on $S$, denoting by $c$ its complex structure, the properties of the Schwarzian derivative allow to say that $Schw(f)$ is a $\pi_1(S)$-invariant holomorphic quadratic differential on ($\widetilde S$,c), hence an element of $\HQD(c)$. 
Denoting by $\mathcal P([c])$ the isotopy classes of projective structures with complex structure in $[c]$, the map
\begin{equation}
	\label{eq: fiber parametrization of P(S)}
	\begin{split}
	\mathcal P([c]) &\to \HQD([c])\\
	[(f,\rho)]&\mapsto [Schw(f)] 
	\end{split}\ 
\end{equation}
is a biholomorphism.

In the paper, we will need the notion of osculating map, which we recall here. 

Let $(f,\rho)$ and $(\widehat f, \widehat\rho)$ be two complex projective structures inducing the same complex structure $c$ on $S$. We define the \textbf{osculating map} for $\widehat f$ with respect to $f$ as the map $\Phi_{\bullet}\colon \widetilde S\to \PSL(2,\C)$ uniquely defined by the fact that $\widehat {f}$ and $\Phi_p\circ f$ have the same (holomorphic) 2-jet in $p\in \widetilde S$, namely
\begin{align*}
\widehat{f}(p)&=\Phi_p(f(p))\\
\widehat{f}\ '(p) &= (\Phi_p \circ f)'(p)\\
\widehat{f}\ ''(p) &= (\Phi_p \circ f)''(p)
\end{align*}
where the derivatives are taken with respect to any local $c$-holomorphic chart on $S$ and any complex affine chart on $\CP$ containing $\widehat f(p)$. Equivalently, $\Phi_p$ can be characterized as having the same 2-jet as $\widehat f\circ f^{-1}$ in a neighborhood of $f(p)$ over which $f$ admits a local inverse. 
One can easily prove that, for all $\gamma\in \pi_1(S)$, $\Phi_{\gamma(p)}\circ \rho(\gamma)= \widehat \rho(\gamma) \circ \Phi_p$ by applying the definition of $\Phi$ to show that the right-hand and left-hand sides are M\"obius maps with the same 2-jet in $p$. The derivative of the osculating map contains the data of the variation of the Schwarzian derivative in the following sense. For all $p\in \widetilde S$, $(\Phi(p))^{-1} (d\Phi|_{p})$ is a linear map from $\C T_pS$ to the Lie algebra of $\PSL(2,\C)$, and, by taking a complex affine chart for $\CP$ around $f(p)$ to see $f$ as a $\C$-valued local holomorphic chart, we locally have 
\begin{equation}
 \label{eq: derivative osculating}
(\Phi(p))^{-1} (d\Phi|_{p})= \frac 1 2 \begin{pmatrix}
    - f & f^2 \\
    -1 & f
\end{pmatrix} \phee df\ ,
\end{equation}
where $\phee (df)^2=Schw(\widehat f)-Schw(f)$ is the difference between the Schwarzian derivatives.

\subsection{Quasi-Fuchsian representations and Bers' Simultaneous Uniformization}
\label{introsection QF}

Let $\rho\colon \pi_1(S)\to \PSL(2,\C)$ be a representation. We define its \textbf{limit set} $\Lambda_\rho$ as the topological boundary of the orbit of any point in $\CP$. 

We say that a discrete and faithful representation $\rho\colon \pi_1(S)\to \PSL(2,\C)$ is \textbf{quasi-Fuchsian} if its limit set is a Jordan curve on $\CP$. The \textbf{quasi-Fuchsian space} is the space of conjugacy classes of quasi-Fuchsian representations, namely
\[
\QF(S)=\faktor{\{\rho\colon \pi_1(S)\to \PSL(2,\C) \text{ quasi-Fuchsian representation}\}}{\PSL(2,\C)}
\]
which is an open smooth subset of the character variety $\mathcal X(\pi_1(S),\PSL(2,\C))$, hence a smooth complex manifold.

Let $\rho$ be a quasi-Fuchsian representation. The orientation on $S$ induces an orientation on $\partial \pi_1(S)$ and hence on $\Lambda_\rho$, which in turn gives a natural orientation to the connected components of $\CP\setminus\Lambda_\rho$: one of the two discs, that we name $\Omega_\rho^+$, is therefore identified with the complex disk $\mathbb D$, while the other one, that we name $\Omega_\rho^-$, is identified with $\overline{\mathbb D}$, the complex disk with opposite orientation. The representation $\rho$ defines a free and properly discontinuous action of $\pi_1(S)$ on both $\Omega_\rho^+$ and $\Omega_\rho^-$, defining two Riemann surfaces on $S$ and $\overline S$ respectively. 

Bers' Simultaneous Uniformization Theorem, as stated in Theorem \ref{Bers Theorem}, allows to build a diffeomorphism
\begin{equation}
	\label{eq: Bers intro}
\mathfrak B\colon \mathcal T(S)\times\mathcal T(\overline S)\xrightarrow{\sim} \QF(S),
 \end{equation}
(that in the rest of the paper will be often seen as an identification) associating to every pair $([c_1],[\overline{c_2}])$ a unique $[\rho]\in \QF(S)$ such that there exists a (essentially unique) projective structure $[(\sigma_1,\rho)]\in \mathcal P(S)$ with $\sigma_1(\widetilde S)=\Omega_\rho^+$ and induced complex structure $[c_1]$, and a (essentially unique) projective structure $[(\overline{\sigma_2}, \rho)]\in \mathcal P(\overline S)$ with $\overline{\sigma_2}(\widetilde S)=\Omega_\rho^-$ and induced complex structure $[\overline {c_2}]$. 
The diffeomorphism $\mathfrak B$ in Equation \eqref{eq: Bers intro} is in fact a biholomorphism in the following sense. The map $\mathfrak B$ maps each $\mathcal T(S)\times \{[\overline{c_2}]\}$ (resp. each $\{[c_1]\}\times\mathcal T(\overline S)$) to a complex submanifold.  The complex structure on $\mathcal T(S)$ (resp. $\mathcal T(\overline S)$) defined by pull-back by $\mathfrak B$ is independent from the choice of $[\overline{c_2}]\in \mathcal T(\overline S)$ (resp. $[c_1]\in \mathcal T(S)$), and it is consistent with the almost complex structure on $\mathcal T(S)$ mentioned in the end of Section \ref{introsection T and T*}. 

We call \textbf{Bers slices} both the subsets of the form
\begin{align*}
	\QF([c_1],\bullet):=\{[c_1]\}\times \mathcal T(\overline S)\ ,\\
	\QF(\bullet, [\overline{c_2}]):=\mathcal T(S)\times\{[\overline{c_2}]\}\ ,
\end{align*} and their images through $\mathfrak B$.

The projective structures $[(\sigma_1,\rho)]$ and $[(\overline{\sigma_2}, \rho)]$ defined from $\mathfrak B^{-1}([\rho])$ are called \textbf{Bers projective structures}.

Teichmüller space $\mathcal T(S)$ is biholomorphic to an open subset of $\mathbb C^{3\mathrm g-3}$. With respect to this complex structure, the projections $\mathrm{Belt(S)}\to \mathcal T(S)$ (with fibers $\mathrm{Belt([c])}$) and $\HQD(S)\to\mathcal T(S)$ have a natural structure of holomorphic vector bundles and are isomorphic to the holomorphic tangent and cotangent bundles respectively. The forgetful map $\mathcal P(S)\to \mathcal T(S)$ is holomorphic as well, and the action of the mapping class group $MCG(S)$ on $\mathcal T(S)$ is by biholomorphisms.

\subsection{The Weil-Petersson metric}
\label{introsection: WP metric}

The description through harmonic Beltrami differentials allows to define a Riemannian metric on Teichmüller space $\mathcal T(S)$, called the \textbf{Weil-Petersson metric}, as follows. Let $z$ be a global holomorphic coordinate for $c$ on $\widetilde S$, $q_1=\phee_1 dz^2$ and $q_2=\phee_2 dz^2$ denote elements in $\HQD(c)$ and $h_0=\varrho_0 dzd\overline z$ be the hyperbolic metric in the conformal class $c$. Then, for all $\left[\frac{\overline{q_1}}{h_0}\right], \left[\frac{\overline{q_2}}{h_0}\right]\in \mathrm{Belt}([c])$ the Weil-Petersson Riemannian metric is defined as
\[\langle{\left[\frac{\overline{q_1}}{h_0}\right], \left[\frac{\overline{q_2}}{h_0}\right]}\rangle _{WP} := Re\left(  q_1\left( \frac{\overline{q_2}}{h_0}\right)\right) =Re\left(\int_S \frac{ \phee_1 \cdot \overline{\phee_2}}{\varrho_0}\frac i 2 dz\wedge d\overline z\right) \ .
\]

The Weil-Petersson metric $\inners_{WP}$ and the complex structure determine on $\mathcal T(S)$ a $MCG(S)$-invariant Kähler manifold structure, together with the symplectic form
\[
\omega \left(\left[\frac{\overline{q_1}}{h_0}\right], \left[\frac{\overline{q_2}}{h_0}\right]\right):= \inner{\left[i\frac{\overline{q_1}}{h_0}\right], \left[\frac{\overline{q_2}}{h_0}\right]}_{WP} = Im\left(\int_S \frac{ \phee_1 \cdot \overline{\phee_2}}{\varrho_0}\frac i 2 dz\wedge d\overline z\right) \ .
\]

In the identification of $\mathcal T(S)$ with $\mathrm{Fuch}(S)$, the symplectic form $\omega$ coincides up to a multiplicative factor 8 with the \textbf{Goldman symplectic form} $\omega_G=8\omega$, a natural symplectic form arising in the character variety of a wide class of real and complex Lie groups $G$ (see \cite{GoldmanSymp}). When $G$ is a complex Lie group, the Goldman symplectic form is $\C$-bilinear and holomorphic with respect to the natural complex structure on the character variety. This class of Lie groups includes $\PSL(2,\C)$, and the restriction of its Goldman symplectic form to $\mathrm{Fuch}(S)$ coincides with the Goldman symplectic form for $\PSL(2,\R)$.

\subsection{Holomorphic Riemannian metrics}
\label{introsection: holo Riemannian metrics} 

We now recall some essential aspects of holomorphic Riemannian metrics, you can find a more detailed treatment in \cite{MeBonsante}.

Let $\mathbb M$ be a complex manifold. \textbf{Holomorphic Riemannian metrics} on complex manifolds can be seen as a natural $\C$-bilinear analog of pseudo-Riemannian metrics: in fact, a holomorphic Riemannian metric on $\mathbb M$ is a nowhere-degenerate holomorphic section of the space of symmetric $\mathbb C$-bilinear forms on the holomorphic tangent bundle. 

A trivial example is $\mathbb{C}^N$ with the holomorphic Riemannian metric on $T\mathbb C^N \cong \mathbb C^N \times \mathbb C^N$ defined by \[\inner{\underline v, \underline w}_{\mathbb C^N}=\sum_{k=1}^N v_k w_k.\] To see a couple of less trivial examples, observe that the metric above descends to the holomorphic Riemannian metric on the compact manifold $\faktor{\mathbb C^n}{\mathbb Z^n+i\mathbb Z^n}$, or consider the complex Killing form on a complex semisimple Lie group (see for instance \cite{MeBonsante} for $\SL(2,\C)$). Another example that we will mention in the paper is given by the space of oriented geodesics of $\Hyp^3$, which admits a unique $\PSL(2,\C)$-invariant holomorphic Riemannian metric of constant curvature $-1$ (see Section \ref{subsection hyperbolic complex metrics} and \cite{MeBonsante}).

The real and the imaginary part of a holomorphic Riemannian metric on $\mathbb M=\mathbb M^n$ are pseudo-Riemannian metrics of signature $(n,n)$. To see this, observe that if $(e_1,\dots e_n)$ is an orthonormal basis for a holomorphic Riemannian metric $\inners$ on $T_p\mathbb M$, then the bases $(e_1,\dots, e_n, ie_1, \dots, ie_n)$ and $(\sqrt i e_1, \dots, \sqrt i e_n, i\sqrt i e_1, \dots, i\sqrt i e_n)$ are orthonormal for $Re\inners$ and $Im\inners$ respectively.

Holomorphic Riemannian metrics come with a notion of \emph{Levi-Civita connection}: namely there exists a unique affine connection, 
\begin{align*}
D\colon \Gamma(T\mathbb M)&\to \Gamma(End_{\mathbb R}(T\mathbb M))\\
X&\mapsto D X
\end{align*}
being torsion free and compatible with the metric, namely $d\inner{X,Y}= \inner{DX, Y}+\inner{X, DY}$. The affine connection $D$ coincides with the Levi-Civita connection of both the real and the imaginary part of the metric.

The metric $\inners$ and its Levi-Civita connection determine a \emph{Riemann curvature tensor} $\mathrm R(X,Y,Z, W)= \inner{D_X D_Y Z -D_YD_X Z- D_{[X,Y]} Z, W}$ which is $\C$-multilinear.

Finally, let $V<T_p \mathbb M$ be a complex vector space with $\dim_\C V=2$ and such that the restriction of $\inners$ to $V$ is non-degenerate, namely there exist two $\C$-linearly independent vectors $X,Y\in V$ such that $\inner{X,Y}\ne 0$. Then, we can define the \emph{sectional curvature} $\mathrm K(V)$ as $\mathrm K(V):=\frac{\mathrm R(X,Y,Y,X)}{\inner{X,X}\inner{Y,Y}-\inner{X,Y}^2}$, whose definition is independent from the choice of the linearly independent vectors  $X,Y$ in $V$.

Similarly as for pseudo-Riemannian metrics, for all $n\ge 2$ and $k\in \mathbb C$ there exists a unique $n$-dimensional simply-connected, complete holomorphic Riemannian manifold of constant sectional curvature $k$ (\cite{MeBonsante}).

\section{Bers metrics}
\label{sec Bers metrics}

\subsection{Hyperbolic complex metrics and Bers Theorem}
\label{subsection hyperbolic complex metrics}
As we already mentioned, throughout the paper, we will assume that $S$ is a closed connected oriented surface of genus greater than or equal to $2$. We denote by $\overline S$ the surface with opposite orientation.

We need to recall some essential technical tools which will be used in some of the proofs of this paper. Most of the content of this subsection follows from \cite{MeBonsante}, with a few changes in the notation.
\vspace{5pt}

A \textbf{complex metric} on $S$ is a smooth section $g$ of $Sym_2(TS) + i Sym_2(TS)$ such that $g$ is non-degenerate, namely for all $p\in S$ the $\mathbb C$-bilinear extension of $g_p$ to the complexified tangent space $g_p\colon \mathbb C T_pS\times \mathbb C T_pS\to \mathbb C$ is a non-degenerate symmetric bilinear form. This notion includes Riemannian metrics on $S$.
We endow the set of complex metrics on $S$ with the $C^\infty$ topology.

With a simple linear algebra argument, one has that for each $p\in S$ the set $\{X\in\mathbb C T_pS\ |\ g_p(X,X)=0\}$ is the union of two complex lines in $\mathbb C T_pS$, corresponding to two points of 
$\mathbb P(\mathbb C T_pS )$. We will call these two points the \textbf{isotropic} directions of $g$, and we call isotropic vectors the vectors $X\in \C TS$ such that $g(X,X)=0$.

Observe that, for each $p\in S$, $\mathbb P(\mathbb C T_pS)\setminus \mathbb P (T_pS)$ is homeomorphic to a 2-sphere from which a great circle is removed, so it has two connected components homeomorphic to disks.

A \textbf{positive complex metric}  is a complex  metric $g$ such that:
\begin{itemize}[noitemsep,topsep=0pt]
\item there are no non-zero isotropic vectors on $TS$, namely $g(v,v)=0$ with $v\in TS$ if and only if $v=0$;
\item the two isotropic directions of $g$ in $\mathbb P (\mathbb C T_p S)$ are points in $\mathbb P(\mathbb C T_pS)\setminus \mathbb P(T_pS)$  that lie in different connected components.
\end{itemize}

One can see that a Riemannian metric $g$ is a positive complex metric: denoting by $z$ a complex coordinate around $p\in S$ for $g$, its isotropic directions are $Span_\C(\partial_z)$ and $Span_\C(\partial_{\overline z})$, which lie in different connected components of $\mathbb P(\mathbb C T_pS)\setminus \mathbb P(T_pS)$.

\vspace{5pt}

Every  complex metric $g$ has a unique \textbf{Levi-Civita connection}, namely a torsion-free affine connection $\nabla\colon \Gamma(\mathbb C TS)\to \Gamma( End(\mathbb C TS) )$ such that 
\begin{equation}
	\label{eq: metric compatibility complex metrics}
	d_X (g(Y, Z) )= g(\nabla_X Y, Z) + g(Y, \nabla_X Z)
\end{equation}
 for all $X, Y, Z\in \Gamma(\mathbb C TS)$. This induces the definition of a $\mathbb C$-multilinear curvature tensor $\mathrm R_g$ and a curvature $\mathrm K_g\colon S\to \mathbb C$.
\vspace{5pt}

Complex metrics of constant curvature $\mathrm K_g=-1$ have a particularly interesting geometric meaning in terms of immersions inside $\mathbb G=\CP\times \CP\setminus \Delta$, which can be interpreted as the space of (maximal, oriented, unparametrized) geodesics of $\Hyp^3$.

The complex manifold $\GG$ can be equipped with a holomorphic Riemannian metric (see Section \ref{introsection: holo Riemannian metrics}) that can be defined as follows: let $(U,z)$ be an affine chart for $\CP$, then in the chart $(U\times U\setminus \Delta, z\times z=(z_1,z_2))$ the metric can be written as 
\[
\inners_{\GG}= -\frac{4}{(z_1-z_2)^2}dz_1\cdot dz_2\ .
\]
This description is independent from the affine chart $(U,z)$. This holomorphic Riemannian metric is invariant under the diagonal action of $\PSL(2,\mathbb C)$ on $\GG$ and has constant sectional curvature $-1$. The isometry group $\Isom(\GG)$ of $\GG$ is generated by the diagonal action of $\PSL(2,\mathbb C)$ and by the diagonal swap $sw\colon(z_1, z_2)\mapsto (z_2, z_1)$. 

{We say that an immersion $\sigma\colon\widetilde S\to \mathbb G$ is totally real if, for all $p\in \widetilde S$, $Span_{\C} (d\sigma(T_p\widetilde S))= T_{\sigma(p)}\mathbb G$.}

\begin{theorem}[\cite{MeBonsante}]
	\label{thm: GC in G}
   {A complex metric $g$ on $\widetilde S$ has constant curvature $-1$ if and only if there exists a totally real immersion $\sigma\colon \widetilde S\to \mathbb G$ such that $g=\sigma^*\inners_{\mathbb G}$. This map $\sigma$ is unique up to post-composition with an element of $\Isom(\GG)$. 
    
    The metric $g$ is $\pi_1(S)$-invariant and defines a metric on $S$ if and only if $\sigma$ is $(\pi_1(S), \Isom(\GG))$-equivariant. 
    
    Finally, the metric $g$ is $\pi_1(S)$-invariant and positive if and only if $\sigma$ can be chosen to be $(\pi_1(S),\PSL(2,\C))$-equivariant and of the form $\sigma=(\sigma_1,\overline{\sigma_2})$, where $\sigma_1, \overline{\sigma_2}\colon \widetilde S\to \CP$ are, respectively, an orientation-preserving and an orientation-reversing local diffeomorphism.}
\end{theorem}

A class of equivariant immersions of $\widetilde S$ into $\GG$ is suggested by Bers' Simultaneous Uniformization Theorem. In the rest of the paper, we will use the notation as in its statement in Theorem \ref{Bers Theorem}.

Each $(c_1,\overline{c_2})\in\CSCS$ defines therefore a $\rho$-equivariant immersion 
\[
(\sigma_1\ccpair, \overline{\sigma_2}\ccpair)\colon \widetilde S \to \CP\times \CP\setminus \Delta=:\GG\ ,
\]
therefore the pull-back positive complex metric can be locally written as
\[
g(c_1, \overline{c_2})=(\sigma_1,\overline{\sigma_2})^*(\inners_\GG) = -\frac 4 {(\sigma_1-\overline{\sigma_2})^2} d\sigma_1\cdot d\overline{\sigma_2}\ ,
\]
defining a map
	\begin{equation}
	\label{eq: parametrization hyperbolic complex metrics}
	\begin{split}
	g\colon	\CSCS &\to \{\text{positive complex metrics}\} \\
		\ccpair&\mapsto g\ccpair :=  -\frac{4}{(\sigma_1- \overline{\sigma_2})^2} d\sigma_1\cdot d\overline {\sigma_2}
	\end{split}
\end{equation}
that is continuous with respect to the $C^{\infty}$ topology (in fact Fréchet-holomorphic as shown in \cite{ESholodependence})
\begin{defi}
	We call $\textbf{Bers metrics}$ the positive complex metrics obtained by this construction, namely the ones in the image of Equation $\eqref{eq: parametrization hyperbolic complex metrics}$.
	
	The $\textbf{holonomy}$ of the Bers metric $g=g(c_1,\overline{c_2})$ is the element $[g]:=\ccpairclass\in \QF(S)$. Equivalently it is the holonomy of the corresponding equivariant isometric immersion $\sigma\colon \widetilde S\to\mathbb G$ as in Theorem \ref{thm: GC in G}.
\end{defi}

Observe that, if $\overline{\sigma_2}=\overline{\sigma_1}$, with $\sigma_1$ onto the upper half plane $H\subset \mathbb C$, then  \[g(c_1, \overline {c_1})= \frac 1 {(Im(\sigma_1))^2}d\sigma_1\cdot d\overline{\sigma_1}\] is the hyperbolic Riemannian metric on $S$ in the conformal class of $c_1$.

\begin{theorem}[\cite{MeBonsante}]
	\label{thm: Bers is connected component and c+c-}
Every positive complex metric is conformal, through a unique factor $\lambda\colon S\to \mathbb C^*$, to a unique Bers metric, determining a continuous projection
\begin{equation}
\label{eq: mappa c+c-}
\begin{split}
(c_+, c_-)\colon \{\text{positive complex metrics}\} &\to \CSCS\\
\lambda \cdot g\ccpair \mapsto \ccpair\ .
\end{split}
\end{equation}

Moreover, the space of Bers metrics is an open connected component of the space of positive complex metrics of constant curvature $-1$, which in turn is an open subset of the space of complex metrics of constant curvature $-1$.
\end{theorem}

Let us now make a few remarks about Bers metrics that will help us handle them.

\begin{remark}
	\label{rmk: from g to c1c2}
	One can see the map $(c_+, c_-)$ from the isotropic directions of the $\mathbb C$-bilinear extension to $\mathbb C TS$, in the following way.
	
	Assume $v\in \mathbb CT_pS$, then observe that 
	\begin{align*}
		g(v,v) =0 &\iff -\frac{4}{(\sigma_1-\overline {\sigma_2})^2}d\sigma_1(v)d\overline{\sigma_2}(v)=0 \iff \\
		& \iff d\sigma_1(v)=0 \quad \text{or}\quad d\overline{\sigma_2}(v)=0 \ .
	\end{align*}
	Denote by $z$ and $\overline w$ the local coordinates for $c_1$ and $\overline {c_2}$ respectively. Then $d\sigma_1(v)=0$ if and only if $v\in Span_{\mathbb C} (\partial _{\overline z} )$, and $d\overline{\sigma_2}(v)=0$ if and only if $v\in Span_{\mathbb C} (\partial _{w} )$.
	
	As a result, $g(v,v)=0$ if and only if $v\in Span_{\mathbb C}(\partial _{\overline z}) \cup Span_{\mathbb C} (\partial_{w} )$.  
	
	So the two complex structures $c_1$ and $\overline {c_2}$ correspond to the unique almost-complex structures $J_1, \overline{J_2}\in End(TS)$ whose $(-i)$-eigenspaces are the two isotropic directions of $g$. 

\end{remark}

\begin{remark}
		\label{rmk: coniugato Bers metric}
		Given a Bers metric $g=g(c_1,\overline{c_2})$, the complex metric $\overline g$, defined by $\overline g (v,w)= \overline{g(v,w)}$ for all $v,w\in TS$, is a Bers metric too. We can see this as follows. 
		
		Assume $g$ corresponds to the immersion $(\sigma_1, \overline {\sigma_2})\colon \widetilde S\to \GG$ and can therefore be written as
	\[
	g=-\frac 4 {(\sigma_1-\overline{\sigma_2})^2} d\sigma_1\cdot d\overline{\sigma_2}\ ,
	\]
	where, with a little abuse, we identify here $\sigma_1$ and $\overline{\sigma_2}$ with their composition with any affine chart on $\CP$. Then,
	\[
	\overline{g}= -\frac 4 {(\sigma_2-\overline{\sigma_1})^2} d\sigma_2\cdot  d\overline{\sigma_1}\ ,
	\]
	which means that $\overline g$ is the complex metric of constant curvature $-1$ corrisponding to the immersion $(s\circ \overline{\sigma_2}, s\circ \sigma_1)$, where $s$ is any orientation reversing involution of $\Hyp^3$ (e.g. the symmetry with respect to a totally geodesic submanifold $\Hyp^2\subset \Hyp^3$). If $\sigma_1$ and $\overline{\sigma_2}$ are $\rho$-equivariant, then $s\circ \overline{\sigma_2}$, and $s\circ \sigma_1$ are $(s\circ \rho\circ s)$-equivariant, with $s\circ \rho\circ s$ being quasi-Fuchsian since its limit set is the image through $s$ of the limit set of $\rho$.
	
	One finally observes that if the data $(c_1, \overline{c_2})\in T(S)$ produces, through Bers theorem, the holonomy $\rho$ and the embeddings $\sigma_1$ and $\overline {\sigma_2}$, then, by uniqueness, the data $(c_2, \overline{c_1})$ corresponds to the holonomy $s\circ \rho \circ s$ and to the embeddings $s\circ \overline{\sigma_2}, s\circ \sigma_1$. We therefore conclude that $\overline g$ is a Bers metric with $\overline g= g(c_2, \overline{c_1})$

	Finally, we remark that the $\mathbb C$-bilinear extension of $\overline g$ to $\mathbb C TS$ satisfies $\overline g(X,Y)= \overline{g(\overline X, \overline Y)}$, for all $X,Y\in \mathbb C TS$.
\end{remark}

\begin{remark}
	\label{rmk: g rhodzdw}
	Let $g$ be a positive complex metric, and let $z$ and $\overline w$ be local coordinates for $c_+(g)$ and $c_-(g)$ on an open subset $U\subset S$, then we can write $g=\varrho dz d\overline w$, with $\varrho\colon U\to \mathbb C^*$.
\end{remark}

\begin{remark}
	\label{rmk: area form}
	As for Riemannian metrics, one can define the area form of $g$ as each of the two nowhere-vanishing 2-forms defined by $\pm\sqrt{g(\partial_{x_1}, \partial_{x_1})g(\partial_{x_2}, \partial_{x_2})- (g(\partial_{x_1}, \partial_{x_2}))^2}\ dx_1\wedge dx_2$, where $(x_1,x_2)$ is any smooth coordinate chart. Since $\CSCS$ is simply connected, there is a consistent choice of an area form for each Bers metric depending on the orientation of $S$: if locally $g=\varrho dzd\overline w$, then its area form consistent with the orientation of $S$ is
	\begin{equation}
		\label{eq: area form of g}
		dA_g= \frac i 2 \varrho dz \wedge d\overline w\ .
	\end{equation}
\end{remark}

Before concluding the section, we introduce a notion for complex metrics that is weaker than positivity. 
\begin{defi}
    We say that a hyperbolic complex metric $g$ on $S$ is \textbf{upper-projective} (resp. \textbf{lower-projective}) if the corresponding equivariant totally real immersion $(\sigma_1, \overline{\sigma_2})\colon \widetilde S\to \mathbb G=\CP\times\CP\setminus \Delta$ is $(\pi_1(S),\PSL(2,\C))$-equivariant and is such that $\sigma_1$ (resp. $\overline{\sigma_2}$) is a local diffeomorphism.

    We say that a complex metric is upper-projective (resp. lower-projective) if it is a conformal multiple of an upper-projective (resp. lower-projective) hyperbolic complex metric. 
\end{defi} 
By Theorem \ref{thm: GC in G}, positive complex metrics are both upper-projective and lower-projective. 

In particular, an upper-projective hyperbolic complex metric $g=(\sigma_1, \overline{\sigma_2})^*\inners_{\GG}$ defines a complex projective structure with developing map $\sigma_1$. 
Generalizing the maps in Equation \eqref{eq: mappa c+c-}, we denote by $c_+(\lambda g)\in \mathcal C(S)$ the complex structure defined by pull-back via the corresponding map $\sigma_1$. Similarly, if $g$ is lower-projective we can define $c_-(\lambda g)\in \mathcal C(\overline S)$ with the pull-back complex structure given by $\overline{\sigma_2}$.

\subsection{Deforming Bers metrics with quadratic differentials}

Recall that a useful notation we use in this paper is to consider, for each $c\in \mathcal C(S)$, a global holomorphic coordinate $z$ on $\widetilde S$ for the lifted complex structure $\widetilde c$. In the same fashion, holomorphic quadratic differentials for $c$ can be seen as $\pi_1(S)$-invariant holomorphic quadratic differentials $\phee dz^2$ on $\widetilde S$.
Also, we will identify complex metrics on $S$ with their lifts to $\widetilde S$, so, when $g$ is positive, we can see $g=\varrho dz d\overline w$ (see Remark \ref{rmk: g rhodzdw}) with $z$ and $\overline w$ being global coordinates on $\widetilde S$ and $\varrho\colon \widetilde S\to \mathbb C^*$.

{
\begin{prop}
\label{prop: g+q same curvature}
    Let $g$ be an upper-projective complex metric with $c_+(g)=c_1$. 
 
    Then, for all $q_1\in \mathrm{HQD}(c_1)$, $ g+q_1$ is a complex metric with the same curvature as $ g$, namely $\mathrm K_{ g}= \mathrm K_{ g+q_1}$. 

\end{prop}
}

\begin{proof}
	Let $z$ denote a local holomorphic coordinate for $c_1$ and write $q_1=\phee dz^2$. Denote by $\nabla$ the Levi-Civita connection of $g$, and define the $(1,1)$-form $\beta\in End(\mathbb C TS)$ by the condition $q_1=g(\beta \cdot, \cdot)$.
		Notice that $\beta$ is self-adjoint for $g$ since $q_1$ is symmetric. 
		
		\begin{itemize}[leftmargin=*,noitemsep,topsep=0pt]
			\item
		We prove that $\beta$ satisfies Codazzi equation $d^{\nabla} \beta=0$. 
		
		Recall that, by Remark \ref{rmk: from g to c1c2}, $g(\partial_{\overline z}, \partial_{\overline z})=0$, and, since the bilinear form $g$ is non-degenerate on $\C TS$, $g(\beta(\partial_{\overline z}), \cdot) = q_1(\partial_{\overline z}, \cdot)=0$ implies that $\beta(\partial_{\overline z})=0$.
		On the other hand, $g(\beta, \partial_{\overline z})=q_1(\cdot,\partial_{\overline z})=0$, so $\beta\in Span_\mathbb C (\partial_{\overline z} \otimes dz)$. 
		Finally, $g(\beta(\partial_{z}),\partial_z )=\phee$ implies $\beta= \frac{\phee}{g(\partial_z, \partial_{\overline z})} \partial_{\overline z} \otimes dz$.
		As a result,
		\begin{align*}
		(d^\nabla \beta)(\partial_{\overline z}, \partial_z)=& \nabla_{\partial_{\overline z} } (\beta(\partial_z) ) - \nabla_{\partial_{z} } (\beta(\partial_{\overline z})) - \beta([\partial_z, \partial_{\overline z}]) =  \nabla_{\partial_{\overline z} } (\beta(\partial_z)) =\\
		&=  {\partial_{\overline z} } \left( \frac{\phee}{g(\partial_z, \partial_{\overline z})}\right)  \partial_{\overline z} + \frac{\phee}{g(\partial_z, \partial_{\overline z})} \nabla_{\partial_{\overline z} } {\partial_{\overline z}} =\\
		&= -\frac{\phee}{(g(\partial_z, \partial_{\overline z}))^2} \cdot \partial_{\overline z} (g(\partial_z, \partial_{\overline z}))  \partial_{\overline z} + \frac{\phee}{g(\partial_z, \partial_{\overline z})} \nabla_{\partial_{\overline z} } {\partial_{\overline z} } \ .	
		\end{align*} 
		By Equation \eqref{eq: metric compatibility complex metrics}, $g(\nabla \partial_{\overline z}, \partial_{\overline z})=\frac 1 2 d(g(\partial_{\overline z}, \partial_{\overline z}))=0$, so $\nabla_v\partial_{\overline z}\in  Span_{\mathbb C}(\partial_{\overline z})$ for all $v\in \C TS$. Therefore, $(d^\nabla \beta)(\partial_{\overline z}, \partial_z)\in Span_{\mathbb C}(\partial_{\overline z})$, so  $(d^\nabla \beta)(\partial_{\overline z}, \partial_z)=0$ if and only if $g((d^\nabla \beta)(\partial_{\overline z}, \partial_z), \partial_z)=0$. Finally:
		\begin{align*}
			g((d^\nabla \beta)(\partial_{\overline z}, \partial_z), \partial_z)=& -\frac{\phee}{(g(\partial_z, \partial_{\overline z}))^2} \cdot \partial_{\overline z} (g(\partial_z, \partial_{\overline z}))  g(\partial_{\overline z}, \partial_z) + \frac{\phee}{g(\partial_z, \partial_{\overline z})}g( \nabla_{\partial_{\overline z} } {\partial_{\overline z} }, \partial z )=\\
			&=-\frac{\phee}{g(\partial_z, \partial_{\overline z})} \cdot \partial_{\overline z} (g(\partial_z, \partial_{\overline z}))  +\frac{\phee}{g(\partial_z, \partial_{\overline z})}g( \nabla_{\partial_{\overline z} } {\partial_{\overline z} }, \partial z )=\\
			&= - \frac{\phee}{g(\partial_z, \partial_{\overline z})}  g(\partial_{\overline z}, \nabla_{\partial_{ {\overline z } } } {\partial_z }) =0
		\end{align*}
		where the last steps follow from $\eqref{eq: metric compatibility complex metrics}$ and by the fact that $\nabla$ is torsion free.
		\item Consider $\alpha= id+\frac 1 2 \beta$. We prove that $g(\alpha\cdot, \alpha \cdot)= g+q_1$. 
		
		Indeed, $\beta(\partial_{\overline z})=0$ and $\beta(\partial_z)\in Span_\C(\partial_{\overline z})$ implies that $tr(\beta)=0$ and $det(\beta)=0$, so from its characteristic polynomial we have that $\beta^2=0$, implying that 
		\[
		g(\alpha\cdot, \alpha \cdot)= g(\alpha^2\cdot, \cdot)= g+g(\beta\cdot, \cdot)= g+q_1 \ .
		\]
		
		\item $g+q_1$ is a (nowhere degenerate) complex metric. In fact, denoting locally $g=\varrho dzd\overline w$ and recalling that $\varrho\ne 0$ and $\partial_{\overline z}\overline w\ne0$, for all $a,b\in \mathbb C$
		\[0=(g+q_1)(a\partial_z+b \partial_{\overline z }, \cdot)= \frac 1 2a\varrho \ d\overline w + \frac 1 2 a \varrho \partial_z \overline w \ dz +  a \phee\  dz + \frac 12 b \varrho \partial_{\overline z}\overline w\  dz 
		\]
		implies $a=0$ and $b=0$,  so $g+q_1$ is non-degenerate.
		
		\item We prove that $g+q_1$ has the same curvature as $g$. 
		
		First of all, notice that $d^\nabla \alpha=0$ and that $det(\alpha)= 1$ since $det(\beta)=tr(\beta)=0$. 
		
		Using  $d^\nabla \alpha=0$  one can see that the affine connection on $\C TS$ defined by 
		\[
		\widehat \nabla _X Y = \alpha^{-1}(\nabla_X \alpha(Y) )
		\]
		is compatible with the complex metric $g+q_1=g(\alpha\cdot, \alpha\cdot)$ and is torsion-free, so $\widehat \nabla$ is the Levi-Civita connection of $g+q_1$. As a consequence, one can immediately check that the corresponding curvature tensor $\widehat {\mathrm{R}}$ is related to the curvature tensor $\mathrm R$ of $\nabla$ via
		\[
		\widehat {\mathrm R}(X,Y) Z =  \alpha^{-1} (\mathrm R(X,Y) \alpha(Z))\ .
		\]
		Finally, using the standard symmetries of the curvature tensor, the curvature $\widehat{\mathrm{K}}$ of $g+q_1$ is such that
		\begin{align*}
		{\mathrm K}_{g+q_1}=&\frac{(g+q_1) (\widehat {\mathrm R}(X,Y)Y, X  ) }{(g+q_1)(X,X)\cdot (g+q_1)(Y,Y) - ( (g+q_1)(X,Y) )^2 }=\\
		&= \frac {g (\mathrm R(X,Y) \alpha(Y), \alpha (X)  ) }{g(\alpha(X),\alpha(X) )g(\alpha(Y),\alpha(Y)) - (g(\alpha(X),\alpha(Y)) )^2 }=\\
		&= \frac 1 {\det(\alpha)^2} \frac{(\det(\alpha))g(\mathrm R(X,Y)Y, X)}{ \left( {g(X,X)g(Y,Y) - (g(X,Y) )^2 } \right)} =\mathrm{K}_h .
		\end{align*} 

    \end{itemize} 
    \end{proof}

{
\begin{prop}
\label{prop: star-shaped}
    Let $g$ be a positive complex metric with $c_1= c_+(g)$. The subset $U_g=\{q_1\in  \HQD(c_1)\ |\ \text{$g+q_1$ is a positive complex metric} \}$ is an open star-shaped subset centered in $0$ of $\HQD(c_1)$. Moreover, if $g$ is a Bers metric, then $g+q_1$ is a Bers metric if and only if it is a positive metric.
\end{prop}

}
\begin{proof}
The fact that $U_g$ is an open subset is a direct consequence of the positivity being an open condition. 
We prove that $U_g$ is star-shaped at $0\in \HQD(c_1)$.

	Let $q_1\in \HQD(c_1)$, $q_1\ne 0$, and let $z$ and $\overline w$ be local coordinates for $c_1$ and $\overline{c
_2}= c_-(g)$.

A quick calculation shows that the isotropic directions of $g+tq_1$ are given by $Span_{\C}(\partial_{\overline z})$ and by the complex span of the vector field
	\[
	W_t:=\partial_w - t \frac{q_1(\partial_w, \partial_w)}{2g(\partial_w, \partial_{ {\overline z } })} \partial_{ {\overline z } }\  .
	\] 

   Therefore, $g+tq_1$ is a positive complex metric if and only if, for all $p\in S$, $Span_\C(W_t)\in  \mathbb P(\mathbb CT_p S)$ lies in the connected component of $ \mathbb P(\mathbb CT_p S)\setminus \mathbb P(T_pS)$ which does not contain $Span_\C(\partial_{ {\overline z } })$, which we call $D_p$.

    By means of the stereographic projection, one can see that for all $p\in S$ the map 
	\begin{align*}
	\mathbb R &\to  \mathbb P(\mathbb C T_pS)\\
	t&\mapsto [W_t]_p
	\end{align*}
is a monotone parametrization of a great circle minus a point, namely $\{[a \frac{q_1(\partial w, \partial w)}{2g(\partial_w, \partial_{ {\overline z } })}\partial_{\overline z} + b\partial_w]\ |\ a,b\in \mathbb R\ ,b\ne 0\}$. As a consequence, 
for each $p\in S$, the condition $Span_{\C}( W_t)\in D_p$ is satisfied for $t$ lying in an interval of $\mathbb R$ containing $0$: thus, by taking the intersection for all $p\in S$, we deduce that the set of $t\in \mathbb R$ for which $g+tq_1$ is a positive complex metric is an interval containg $0$. 
 
 If $g$ is a Bers metric, then, since by Theorem \ref{thm: Bers is connected component and c+c-} Bers metrics define a connected component of the space of positive hyperbolic complex metrics, the fact that $U_g$ is connected allows us to conclude that $g+q_1$ is Bers if and only if it positive.

\end{proof}

\begin{remark}
    Analog results to Propositions \ref{prop: g+q same curvature} and \ref{prop: star-shaped} hold for lower-projective complex metrics $g$ and for $c_-(g)$-holomorphic quadratic differentials. A simple way to prove this follows from observing that $g$ is lower-projective if and only if $\overline g$ is lower-projective and $\mathrm K_{\overline g}= \overline{\mathrm K_g}$. When $g$ is a positive complex metric, we denote by $V_g$ the star-shaped open subset of $\HQD(c_-(g))$ analog to $U_g$.
\end{remark}

\begin{remark}
	\label{rmk: conti su h+q}
	Let $g$ be a positive complex metric, let $\ccpair=(c_+,c_-)(g)$, and let $U_g$ be as in Proposition \ref{prop: star-shaped}. There are a few interesting computations we can make to understand the new Bers metric $g+q_1$.
	
	Lift the tensors to the universal cover $\widetilde S$, denote $g=\varrho_0 dz d\overline w$ and $q_1=\phee dz^2\in U_g$.
	
	By construction of Bers metrics, there exists a smooth function $\varrho\colon \widetilde S\to {\mathbb C}^*$, and an orientation-reversing $(\pi_1(S),\PSL(2,\C))$-equivariant map $\overline \eta\colon \widetilde S\to \CP$ with open image and which is a diffeomorphism onto its image, such that
	\[
	\varrho_0 dz d\overline w +\phee dz^2 = g+q_1 = \varrho dz d\overline \eta \ .
	\]
	Since the symmetric product between 1-forms is non-degenerate, we get that
	\begin{equation}
		\label{eq: 1 forme in h+q}
		\varrho d\overline \eta= \varrho_0 d\overline w +\phee dz\ .
	\end{equation}
	
	From Equation \eqref{eq: 1 forme in h+q}, we get a few interesting remarks:
	\begin{enumerate}
		\item \label{item 1} $\varrho= \varrho_0 \frac{\partial_{\overline z} \overline w}{\partial_{\overline z} \overline \eta}$ and $\varrho \partial_z \overline \eta= \varrho_0\partial_z\overline w+\phee $ (from evaluating in $\partial_{\overline z}$ and in $\partial_z$ respectively).
		\item \label{item 2} $\varrho {\partial_w \overline \eta} =\phee { \partial_w z}$ and $\varrho \partial_{\overline w}\overline{\eta}=\varrho_0+\phee \partial_{\overline w} z$ (from evaluating in $\partial_{w}$ and in $\partial_{\overline w}$ respectively).
		\item \label{item 3} $\frac{\partial_z \overline{\eta}} {\partial_{\overline z}\overline{\eta} }= \frac{\varrho_0 \partial_z\overline w+\phee}{\varrho_0 \partial_{\overline z}{\overline w}}$ (from item \eqref{item 1}).
		\item \label{item 4} $\frac{\partial_{w} \overline \eta}{\partial_{\overline w} \overline \eta } = \frac{\phee \partial_w z}{\varrho_0 + \phee \partial_{\overline w} z}$  (from item \eqref{item 2}).
	\end{enumerate}
	We get similar calculations from the study of $g+\overline {q_2}$, with $\overline{q_2}=\overline \psi d\overline w^2\in V_g\subset \HQD(\overline{c_2})$.

\end{remark}

\begin{lemma}
	\label{lemma: beltrami differentials}
	Let $g=g(c_1, \overline {c_2})$ be a Bers metric, $z,\overline w$ coordinates for $c_1$, $\overline{c_2}$, and $g=\varrho dzd\overline w$. 
	
	For all $q_1\in \HQD(c_1)$, the infinitesimal deformation of $c_-(g+tq_1) \in \mathcal C(\overline S)$ at $t=0$ is given by the Beltrami differential $ \frac{\phee \partial_w z } {\varrho_0 } \frac{d w}{d \overline w}$, where $q_1=\phee dz^2$.
	
	Similarly, for all $\overline{q_2}\in \HQD(\overline{c_2})$, the infinitesimal deformation of $c_+(g+t\overline{q_2}) \in \mathcal C(S)$ at $t=0$ is given by the Beltrami differential $\frac{\overline \psi \partial_{\overline z} \overline w} {\varrho_0 } \frac{d\overline z}{d z} $, where $\overline{q_2}=\overline \psi d\overline w^2$.
\end{lemma}
\begin{proof}
	Let us reconsider the methods of Remark \ref{rmk: conti su h+q}. 
	For $q_1\in \HQD(c_1)$, $q_1=\phee dz^2$, we have
	\[
	\varrho dzd\overline w +t\phee dz^2=g+tq_1 = \varrho_t dzd\overline {w_t}
	\]
	for some $\pi_1(S)$-equivariant orientation-reversing open embedding $\overline w_t\colon \widetilde S\to \CP$, with $\overline w= \overline w_0$.
	For sufficiently small $t$, we have that $c_-(g+t q_1))\in \mathcal C(\overline S)$ is the complex structure induced by $\overline w_t$.
	
	By item $\eqref{item 4}$ in Remark \ref{rmk: conti su h+q}, we obtain that
	\[
	\frac{\partial_{w} \overline {w_t}\ dw}{\partial_{\overline w} \overline{w_t}  \ d\overline w } = \frac{t\phee \partial_w z}{\varrho_0 + t\phee \partial_{\overline w} z} \frac{dw}{d\overline w} ,  
	\] 
	so we conclude that ${\frac d {dt}}_{|t=0} \frac{\partial_{w} \overline {w_t} d w }{\partial_{\overline w} \overline{w_t} d\overline w  }= \frac{\phee \partial_w z } {\varrho_0 } \frac{d w}{d \overline w} $.
	
	The proof for the case $g+t\overline {q_2}$ is totally symmetric.
\end{proof}

\section{The Schwarzian map and the metric model for $T \QF(S)$}
\label{sec: Metric model}

\subsection{The relation with the Schwarzian derivative}

Let $g$ be an upper-projective hyperbolic complex metric on $S$ with $c_+(g)=c_1$.
In the notation of Theorem $\ref{thm: GC in G}$, let $\sigma=(\sigma_1, \overline{\sigma_2})\colon \widetilde S\to \mathbb G$ be the corresponding equivariant immersion. We define $\mathbf{Schw}_+(g)\in \HQD(c_1)$ to be the Schwarzian derivative of the complex projective structure defined by $\sigma_1$ and by the holonomy of $\sigma$.

The aim of Section \ref{sec: Metric model} is to prove the following Theorem.
\begin{theorem}
	\label{thm: schwartzian}
 Let $g$ be a hyperbolic complex metric on $S$.
	
	If $g$ is upper-projective with $c_+(g)=c_1$, then, for all $q_1\in \HQD(c_1)$, $g+q_1$ is upper-projective with $c_+(g+q_1)=c_1$, and
	\begin{equation}
		\label{eq: schw holo}
	\mathbf{Schw}_+(g+q_1) = \mathbf{Schw}_+(g)- \frac 1 2 q_1 \ .
	\end{equation}
    Similarly, if $g$ is lower-projective with $c_-(g)=\overline{c_2}$, then, for all $\overline{q_2}\in \HQD(\overline{c_2})$, $g+\overline{q_2}$ is lower-projective with $c_-(g+\overline{q_2})=\overline{c_2}$, and
	\begin{equation}
		\label{eq: schw antiholo}
		\mathbf{Schw}_-(g+\overline{q_2}) = \mathbf{Schw}_-(g) -\frac 1 2 \overline{q_2}\ .
	\end{equation}

The maps $q_1\mapsto [g+q_1]$ and $\overline{q_2}\mapsto [g+\overline{q_2}]$ are hence holomorphic and define a model for the holomorphic tangent bundle of $\QF(S)$ given by
\begin{equation}
	\label{eq: tangent model}
	T_{([c_1], [\overline{c_2}])}\QF(S)\cong \HQD(c_1)\oplus \HQD(\overline{c_2}) 
\end{equation}
with the induced almost-complex structure defined by the multiplication times $i=\sqrt{-1}$.
\end{theorem}

We will refer to the model in \eqref{eq: tangent model} as the \textbf{metric model} for the holomorphic tangent bundle of $\QF(S)$.

\begin{cor}
\label{cor: g+q isotopy independent}
Let $\phi\in \mathrm{Diff}_0(S)$ be an isotopy, let $c_1\in \mathcal C(S)$ and $c_1'=\phi^*c_1$. Let $g$ and $g'$ be hyperbolic complex metrics such that $g$ is $c_1$-projective, $g'$ is $c_1'$-projective, and $\phi^*(\mathbf{Schw}_+(g))=\mathbf{Schw}_+(g')$. Then, for all $q_1\in \HQD(c_1)$,  $g+q_1$ and $g' +\phi^*q_1$ induce isotopic complex projective structures for $c_1$ and $c'_1$. In particular, they have the same holonomy.
\end{cor}

Before proving Theorem \ref{thm: schwartzian}, let us see some remarks concerning this result.

\begin{remark}
	The notation $T_{([c_1], [\overline{c_2}])}\QF(S)\cong \HQD(c_1)\oplus \HQD(\overline{c_2})$ can be intended as \[
	\HQD([c_1])\oplus \HQD([\overline{c_2}])=\{q_1+\overline{q_2}\ |\ q_1\in \HQD(c_1),\ \overline{q_2}\in \HQD(\overline{c_2})  \}< \Gamma( Sym_2(TS,\C)) \ .
	\]
	In fact, Bers metrics can be seen as the image of a holomorphic map $g\colon \mathcal C(S)\times\mathcal C(\overline S)\to\Gamma(Sym_2(T S,\C))$ in the sense of Fréchet manifolds (see \cite{ESholodependence}): this map is an immersion, with the image of the differential map in $\ccpair$ being generated by the complex span of Lie derivatives of $g(c_1,\overline{c_2})$ (which correspond to complex Lie derivatives, a notion defined in \cite{ElSa}), and by $\HQD(c_1)\oplus \HQD(\overline{c_2})$, which correspond to infinitesimal deformations of the form $\frac d {dt}(g+tq_1)$ and $\frac d {dt}(g+t\overline{q_2})$. 
 The differential of its projection $\mathcal C(S)\times \mathcal C(\overline S)\to \QF(S)$ maps injectively and surjectively the subspace corresponding to $ \HQD(c_1)\oplus \HQD(\overline{c_2})$ onto $T_{([c_1], [\overline{c_2}])}\QF(S)$ as in Theorem \ref{thm: schwartzian}.
\end{remark}

\begin{remark}
	\label{rmk: models for the tangent}
	By definition, each left Bers slice $\QF(\bullet, [\overline{c_2}])$ is biholomorphic to $\mathcal T(S)$, while each right Bers slice $\QF([c_1], \bullet)$ is biholomorphic to $\mathcal T(\overline S)$: Lemma \ref{lemma: beltrami differentials} gives a description of the differential of these maps, namely
	
	\begin{equation}
		\label{eq: models tangent da QF(c,.)} 
		\begin{split}
		\HQD(c_1)\cong T_{[\overline{c_2}]} \QF([c_1], \bullet)  &\xrightarrow{\sim} T_{[\overline{c_2}]}\mathcal T(S)\cong \mathrm{Belt}(\overline{c_2})\\
		q_1=\phee dz^2&\mapsto \left[ \frac{\phee \partial_w z}{\varrho}\frac{dw}{d\overline w}\right] \ ,
		\end{split}
	\end{equation}
	and 
		\begin{equation}
				\label{eq: models tangent da QF(., c)} 
		\begin{split}
		\HQD(\overline{c_2})\cong T_{[c_1]} \QF(\bullet, [\overline{c_2}])&\xrightarrow{\sim} T_{[{c_1}]} \mathcal T(\overline S)\cong \mathrm{Belt}({c_1})\\
		\overline{q_2}=\overline{\psi} d\overline w^2&\mapsto \left[ \frac{\overline \psi \partial_{\overline z} \overline w}{\varrho}\frac{d\overline z}{d z}\right]
	\end{split}
\end{equation}
\end{remark}

\begin{remark}
	\label{rmk: differential of the fuchsian map}
 With respect to the metric model for the holomorphic tangent bundle of $\QF(S)$ in \eqref{eq: tangent model}, we can write the differential of the Fuchsian map $\mathcal T(S)\to \mathrm{Fuch(S)}\subset \QF(S)$ as:
 \begin{align*}
 T_{[c]}\mathcal T(S)&\xrightarrow{\sim} 	T_{([c],[\overline c])} \mathrm{Fuch}(S)< T_{([c],[\overline c])} \QF(S)\\
 \left[\frac{\overline q}{g_0} \right]&\mapsto q+\overline q
 \end{align*}
 where $g_0=g(c,\overline c)$ and $q\in \HQD(c)$.
 To see this, just observe that the composition with the differential of the projection of $\QF(S)$ to the Bers slice $\QF(\bullet, [\overline c]) =\mathcal T(S)\times\{[\overline{c}]\}$ (resp. $\QF([c], \bullet)=\{[c]\}\times\mathcal T(\overline S)$) gives $[\frac{\overline q}{g_0} ]\mapsto [q]$ (resp.$[\frac{\overline q}{g_0} ]\mapsto [\overline q]$) which coincides with the inverse of map \eqref{eq: models tangent da QF(., c)} (resp. the inverse of the conjugate of map \eqref{eq: models tangent da QF(c,.)}) for $c=c_1=c_2$.

\end{remark}

\begin{remark}
	\label{rmk: metric model Teichmuller}
	This metric model for the holomorphic tangent bundle of $\QF(S)$ extends the metric model for Teichmüller space, identified through the uniformization theorem with the space $Met_{-1}(S)$ of hyperbolic Riemannian metrics up to isotopy, studied in \cite{BergerEbin} and \cite{FischerTromba}. We recall it briefly.
	
	Let $h_0$ be a hyperbolic metric, and let $\dot h$ denote an infinitesimal variation of it with hyperbolic metrics, namely $\dot h={\frac d {dt}}_{|t=0} h_t\in Sym_2(TS)$. Then one can see the tangent bundle of $Met_{-1} (S)$ as
	\begin{equation}
		\label{eq: metric model Teichmuller}
		T_{[h_0]}Met_{-1}(S)=\{ \dot h\in Sym_2(TS)\ |\ tr_{h_0}(\dot h)=0 \ \text{ and } \ div_{h_0} (\dot h)=0 \}\  
	\end{equation}
	where $div$ denotes the divergence. 
	
	A classic computation shows that  $tr_{h_0}(\dot h)=0$ and $div_{h_0} (\dot h)=0$ if and only if $h_0$ is the real part of a holomorphic quadratic differential for the complex structure $c_0$ induced by $h_0$.  We therefore get:
	\[
	T_{[h_0]}Met_{-1}(S)=\{ q+\overline q \ |\ q\in \HQD(c_0) \}\ .
	\]
	which coincides with the tangent bundle to $\mathrm{Fuch}(S)$ seen as a submanifold of $\QF(S)$, as shown in Remark \ref{rmk: differential of the fuchsian map}.
\end{remark}

\subsection{Proof of Theorem \ref{thm: schwartzian}}

We approach the proof of Theorem \ref{thm: schwartzian}. We thank the referee for suggesting this alternative proof strategy.

Let $g$ be an upper-projective hyperbolic complex metric, with $c_+(g)=c_1$ and with corresponding equivariant immersion $\sigma=(\sigma_1, \overline{\sigma_2})\colon \widetilde S\to \mathbb G$, whose holonomy we denote by $\rho_0$. As we observed before, $(\sigma_1, \rho_0)$ defines a $\CP$-structure conformal to $c_1$. 

Let $q_1\in \mathrm{HQD}(c_1)$ and let $(\widehat{\sigma_1}, \rho)$ be the $\CP$-structure uniquely characterized by the fact that it induces the complex structure $c_1$ and that the Schwarzian derivative is $Schw(\widehat{\sigma_1})=Schw(\sigma_1)+q_1$. 

Now, consider the {osculating} map $\Phi_{\bullet}\colon \widetilde S\to \PSL(2,\C)$ for $\widehat{\sigma_1}$ with respect to $\sigma_1$, defined as in Section \ref{introsection proj structures}.  
Since, $\Phi_{\gamma(p)}\circ \rho_0(\gamma)= \rho \circ \Phi_p(\gamma)$ for all $\gamma\in \pi_1(S)$, the map $\widehat{\overline{\sigma_2}}\colon \widetilde S\to \CP$ defined by $\widehat{\overline{\sigma_2}}(p):=\Phi_p({\overline{\sigma_2}}(p))$ is $\rho$-equivariant, and $\widehat{\overline{\sigma_2}}(p)\ne \widehat{\sigma_1}(p)$ for all $p$, therefore $\widehat \sigma:= (\widehat {\sigma_1}, \widehat{\overline{\sigma_2}})$ is a $\rho$-equivariant immersion. Let $\widehat g= \widehat \sigma ^*\inners_{\mathbb G}$.

\begin{prop}
\label{prop: osculating}
    With the notations above, 
    \[
    \widehat g= g-2q_1 \ .
    \]
\end{prop}
\begin{proof}
Define a local holomorphic chart on $\widetilde S$ by taking a local inverse of $\sigma_1$ composed with an affine chart of $\CP$, and denote $q_1= \phee (d\sigma_1)^2$.

    By taking a local lift to $\SL(2,\C)$, denote 
    \[
    \Phi_p= \begin{pmatrix}
        a(p) & b(p)\\
        c(p) & e(p) 
    \end{pmatrix}\ ,
    \]hence, from $\eqref{eq: derivative osculating}$, we get that
    \[
    \begin{cases}
        cda - a dc&= \frac 12 \phee d\sigma_1\\
        cdb-ade&= -\frac 12 \sigma_1 \phee d\sigma_1\\
        eda-bdc&=-\frac 1 2 \sigma_1 \phee d\sigma_1\\
        edb-bde&=\frac 12 \sigma_1^2 \phee d\sigma_1
    \end{cases}\quad .
    \]
    The proof then follows by an explicit computation: 
  \begin{align*}
      -&\frac 4 {(\widehat{\sigma_1}(z)-\widehat{\overline{\sigma_2}}(z))^2} d\widehat{\sigma_1}\cdot\widehat{\overline{\sigma_2}}=\\
      &\quad=-\frac 4 {\left(\frac{a(z){\sigma_1}(z)+b(z)}{c(z)\sigma_1(z)+e(z)} -
      \frac{a(z){\overline{\sigma_2}}(z)+b(z)}{c(z)\overline{\sigma_2}(z)+e(z)}\right)^2 }\ d\left( \frac{a(z){\sigma_1}(z)+b(z)}{c(z)\sigma_1(z)+e(z)}\right) \cdot d \left(\frac{a(z)\overline {\sigma_2}(z)+b(z)}{c(z)\overline{\sigma_2}(z)+e(z)}\right)=\\
      &\quad = -\frac 4 {(\sigma_1-\overline{\sigma_2})^2} d\sigma_1 \cdot \left(d\overline{\sigma_2}+\frac 1 2 \phee (\sigma_1-\overline{\sigma_2})^2 d\sigma_1\right)=\\
      &\quad =-\frac 4 {(\sigma_1-\overline{\sigma_2})^2} d\sigma_1\cdot d\overline{\sigma_2} -2 q_1 \ .
  \end{align*}
\end{proof}

\begin{proof}[Proof of Theorem \ref{thm: schwartzian}] By the construction above and by Proposition \ref{prop: osculating}, $\mathbf{Schw}_+(g+q_1)= \mathbf{Schw}_+(g)-\frac 12 q_1$ for all $q_1\in \HQD(c_1)$, as in the statement.  Since the map $\mathcal P(c_1)\to \HQD(c_1)$ is a biholomorphism, we conclude that the map $q_1\mapsto [g+q_1]$ is holomorphic. The proof for the lower-projective case follows in the very same fashion.

Assume now that $g=g\ccpair$ is a Bers metric and define $U_g$ as in Proposition \ref{prop: star-shaped}. Since Bers projective structures are uniquely determined by their holonomy, the map $U_g\to \QF([c_1], \bullet)$ given by $q_1\mapsto [g+q_1]$ is also injective, hence a biholomorphism because the dimensions coincide. Applying the same argument to the Bers slice $\QF(\bullet, [\overline{c_2}])$, we get the tangent model in Equation $\eqref{eq: tangent model}$.
\end{proof}

\begin{remark}
   Observe that Proposition \ref{prop: osculating} is consistent with Proposition \ref{prop: g+q same curvature}. In fact, the approach suggested by the osculating map can be used to construct a more geometric proof of Proposition \ref{prop: g+q same curvature} for the special case when $g$ has curvature $-1$. More precisely, 
let $g$ denote an upper-projective complex metric of curvature $-1$, and let $\sigma=(\sigma_1, \overline{\sigma_2})\colon \widetilde S\to \mathbb G$ be such that $g=\sigma^*\inners_{\mathbb G}$. For all $q_1\in \mathrm{HQD}(c_+(g))$, construct the osculating map $\Phi$ with respect to $\sigma_1$ for the developing map of the complex projective structure whose Schwarzian differential is $Schw(\sigma_1)-\frac 12 q$. Then, using the same computation as in the proof of Proposition \ref{prop: osculating}, the map $\widetilde S\to \mathbb G$ given by $p\mapsto (\Phi_p(\sigma_1(p)), \Phi_p(\overline{\sigma_2}(p)) )$ is such that the pull-back metric $g'$ satisfies $g'=g+q_1$. Since $g'$ has curvature $-1$ by Theorem \ref{thm: GC in G}, we conclude that $\mathrm K_{g+q_1}=-1=\mathrm K_g$. We thank the referee for this remark.
\end{remark}

\section{The holomorphic extension of the Weil-Petersson metric}
\label{section: sec metric}

\subsection{The bilinear form on $\QF(S)$}
The tools we have developed so far allow us to construct a bilinear form on the holomorphic tangent bundle of $\QF(S)$.

Let $\tau_1, \tau_2\in Sym_2(\mathbb C TS)$. A complex metric $g$ on $S$ induces a bilinear form on $Sym_2(\mathbb C TS)$ as follows. Let $(X_1,X_2)$ be any linear basis for $\mathbb CT_pS$, let $g_{ij} = g(X_i, X_j)$, and $(g^{ij}) = (g_{ij})^{-1}$. Then, for all $\tau_1,\tau_2\in \mathbb C T_p S$, define
\begin{equation}
	\label{eq: metric on symmetric forms}
	<{\tau_1, \tau_2}>_g = \sum_{i,j,k, \ell=1,2} \tau_1(X_i,X_j)\tau_2(X_k, X_\ell) g^{ik} g^{j\ell}
\end{equation}
The definition of $<{\tau_1, \tau_2}>_g$ is independent from the choice of the basis $(X_1, X_2)$. Equivalently, considering the endomorphisms $\tau_1^\sharp, \tau_2^\sharp\in End(\C TS)$ defined by $\tau_k=g(\tau_k^{\sharp}, \cdot)$, we have $<{\tau_1, \tau_2}>_g=tr(\tau_1^{\sharp}\circ \tau_2^{\sharp})$.

\begin{lemma}
	\label{lemma: metric on hqd}
	Let $g=g(c_1, \overline{c_2})$ be a Bers metric, and let $q_1\in \HQD(c_1)$, $\overline{q_2} \in \HQD(\overline{c_2})$. Then,
	\[
	<q_1, \overline{q_2}>_g dA_g= 2i \frac{\phee \overline{\psi}}{\varrho} dz\wedge d\overline{w}\ 
	\]
	where $g=\varrho dzd\overline w$, $q_1=\phee dz^2$, and $\overline{q_2}=\overline{\psi} d\overline w$.
\end{lemma}
\begin{proof}
	Consider on $\widetilde S$ the basis $(\partial_{\overline z},  \partial_w)$. 
	Then, $0=q_1(\partial_{\overline z}, \partial_{\overline z})=g(\partial_{\overline z}, \partial_{\overline z})=g(\partial_w,\partial_w)= \overline{q_2}(\partial_w,\partial_w)$. Recalling the description of the area form as in Equation $\eqref{eq: area form of g}$, we get:
\begin{align*}
	<{q_1, \overline{q_2}}>_g dA_g &=  q_1(\partial_{w}, \partial_{w}) \overline{q_2}(\partial_{\overline z} , \partial_{\overline z} ) \frac{1}{(g(\partial_{\overline z}, \partial_w))^2} dA_g=\\
	&= \frac {\phee (\partial_w z)^2\  \overline{\psi}(\partial_{\overline z} \overline w)^2}{\frac 1 4 \varrho^2 (\partial_{\overline z} \overline w)^2 (\partial_{w}z)^2} \cdot \frac i 2 \varrho dz\wedge d\overline{w}=\\
	&=2i\frac{\phee \overline\psi}{\varrho}dz\wedge d\overline w\ .
\end{align*}
\end{proof}

\begin{remark}
    In the notations above, by Remark \ref{rmk: models for the tangent}, the element $[\mu]\in T_{[c_1]}\mathcal T(S)$ corresponding to $\frac d {dt}_{|t=0} c_+(g+t\overline {q_2})$ is represented by $\mu=\frac{\overline \psi \partial_{\overline z}\overline w}{\varrho}\frac{d\overline z}{dz}$. Therefore, by \eqref{eq: pairing WP},
    \[
    q_1(\mu)=\frac 1 4 \int_S  <q_1, \overline{q_2}>_g dA_g\ .
    \]
\end{remark}

\begin{lemma}
	\label{lemma: isotopy invariance of the integral}
	Let $g=g(c_1, \overline{c_2})$ be a Bers metric, let $\phi_1,\phi_2\in \mathrm{Diff}_0(S)$, and consider $g'= g(\phi_1^*(c_1), \phi_2^*(\overline{c_2}))$. 
	
	Then, the $2$-forms \[	<{q_1, \overline{q_2}}>_g dA_g\qquad  \text{and}\qquad 	<\phi_1^*(q_1), \phi_2^*(\overline{q_2})>_{g'} dA_{g'}\] on $S$ differ by an exact form. 
	In other words,
	\begin{equation}
\label{eq: integral isotopy invariance}
	\int_S \frac{\phee \overline{\psi}}{\varrho} dz\wedge d\overline{w}= \int_S \frac{\phi_1^*(\phee) \phi_2^*(\overline{\psi})}{\varrho'} \phi_1^*(dz)\wedge \phi_2^*(d\overline{w})
	\end{equation}
	where $g'= \varrho' \phi_1^*(dz) \phi_2^*(d\overline{w})$.
\end{lemma}
\begin{proof}
	
		Let $(\sigma_1,\overline \sigma_2)= (\sigma_1 (c_1, \overline{c_2}),\overline {\sigma_2}\ccpair)$ be the pair of developing maps as in Bers Theorem (Theorem \ref{Bers Theorem}).
	Observe that, for all $\phi_1,\phi_2\in \mathrm{Diff}_0(S)$, $\sigma_1 (\phi_1^*(c_1), \phi_2^*(\overline{c_2}))=\sigma_1\circ \phi_1$ and $\overline {\sigma_2}(\phi_1^*(c_1), \phi_2^*(\overline{c_2}))= \overline {\sigma_2}\circ \phi_2$.
	Therefore, considering the global coordinates on $\widetilde S$ given by $z=\sigma_1$, $\overline w=\overline{\sigma_2}$, we have that
	\begin{align*}
g(\phi_1^*(c_1), \phi_2^*(\overline{c_2}))&=-\frac{4}{((z\circ \phi_1)-(\overline w\circ \phi_2))^2} (\phi_1^*dz)\cdot (\phi_2^*d\overline w)\ . 
	\end{align*}
We denote $\varrho(\phi_1, \phi_2)=  -\frac{4}{((z\circ \phi_1)-(\overline w\circ \phi_2))^2}$, so that $\nobreak{g(\phi_1^*(c_1), \phi_2^*(\overline{c_2}))= \varrho(\phi_1, \phi_2)(\phi_1^*dz)\cdot (\phi_2^*d\overline w)}$. 

Now, since pulling back a 2-form through an isotopy does not change its integral, we have that
\begin{align*}
	\int_S \frac{\phi_1^*(\phee) \phi_2^*(\overline{\psi})}{\varrho(\phi_1,\phi_2)} \phi_1^*(dz)\wedge \phi_2^*(d\overline{w})&=
\int_{S} (\phi_1^{-1})^* \left(\frac{(\phee\circ \phi_1) (\overline{\psi}\circ \phi_2)}{\varrho(\phi_1, \phi_2)} \phi_1^*(dz)\wedge \phi_2^*(d\overline{w})\right) =\\
&= \int_S \frac{\phee\cdot (\overline \psi \circ \phi_2\circ\phi_1^{-1} )}{\frac{-4}{(z- (\overline w\circ  \phi_2\circ\phi_1^{-1} ))^2} } dz\wedge  (\phi_2\circ\phi_1^{-1})^*(d\overline w)\\
&= \int_S \frac{\phee\cdot (\overline \psi \circ \phi_2\circ\phi_1^{-1} ) }{\varrho(id, \phi_2\circ\phi_1^{-1})} dz\wedge  (\phi_2\circ\phi_1^{-1})^*(d\overline w)\ .
\end{align*}
	As a result, in the statement of the Lemma we can assume without loss of generality that $\phi_1=id$ and that we are only deforming $\overline {c_2}$ with an isotopy.
	\vspace{3pt}
	
	Also, observe that in order to prove the statement it is sufficient to show the infinitesimal version of it, namely that for all smooth path $\phi_t\in \mathrm{Diff}_0(S)$, $\phi_0=id$, the 2-form 
	\[
	\frac{d}{dt}_{|{t=0}} \left( -\frac 1 4 {\phee\cdot (\overline{\psi}\circ\phi_t) } (z-(\overline w \circ \phi_t))^2 dz\wedge \phi_t^*(d\overline w) \right)
	\]
	is a $\pi_1(S)$-invariant exact form on $\widetilde S$, then the Lemma follows by the generality of the initial $g$, by the fact that the form depends continuously on $t$ and that $S$ is compact, so integration and derivation with respect to $t$ commute.

Denote $X= \frac d {dt}| _{t=0} \phi_t$, $X$ being a tangent $\pi_1(S)$-invariant vector field on $S$, and denote $X=\alpha\partial_w +\beta\partial_{\overline w}$. Using the standard properties of the Lie derivative $\mathcal L$, we have that 
\begin{align*}
		{\frac{d}{dt}}_{|{t=0}} \Big( \phee\cdot (\overline{\psi}\circ\phi_t)  &(z-(\overline w \circ \phi_t))^2 dz\wedge \phi_t^*(d\overline w)\Big)= \\
		=&\phee (\partial_X \overline{\psi}) (z-\overline w)^2 dz\wedge d\overline w- 2\phee\overline{\psi}(z-\overline w)(\partial_X \overline w) dz \wedge d\overline{w} +\\
		&+ \phee \overline \psi (z-\overline w)^2 dz\wedge \mathcal L_X(d\overline w)=\\
		=&\phee dz\wedge\left( \beta (z-\overline w)^2 d\overline \psi - 2 \overline \psi (z-\overline w) \beta d \overline w +\overline \psi (z-\overline w)^2 d\beta \right)=\\
		=&\phee dz \wedge d((z-\overline w)^2 \overline \psi \beta)=\\
		=& - d\left((z-\overline w)^2\phee \overline \psi \beta dz \right)
\end{align*}
were the last step follows from the fact that $\phee dz$ is closed since it is $c_1$-holomorphic.

	One can now easily show that the 1-form $\omega:= (z-\overline w)^2\phee \overline \psi \beta dz $ is $\pi_1(S)$-invariant. This can be seen by observing that, since $g$ is $\pi_1(S)$-invariant, $(z-\overline w)^2\partial_z\otimes \partial_{\overline w} =-\frac 1 {4\varrho} \partial_z\otimes \partial_{\overline w}$ is also $\pi_1(S)$-invariant as well. Moreover, for all $v\in TS$, \[
   ( q_1\otimes \overline {q_2})\left(v\otimes (z-\overline w)^2\partial_z\otimes \partial_{\overline w}\otimes \beta \partial_{\overline w}\right)= \omega(v)\ ,
    \]
we conclude that $\omega$ is $\pi_1(S)$-invariant as well. We thank the referee for this observation.

\end{proof}

By integrating the pairing in Equation \eqref{eq: metric on symmetric forms}, we can define a bilinear form on $\HQD(c_1)\oplus \HQD(\overline{c_2})$ by 
\begin{equation}
	\label{eq: def metrica in generale}
	\inner{\tau_1, \tau_2}_g= \frac 1 {8}\int_S <\tau_1, \tau_2>_g dA_g\ ,
\end{equation}
where $g=g(c_1,\overline{c_2})$ and $dA_g$ is the area form compatible with the orientation as in Equation \eqref{eq: area form of g}.

\begin{prop}
	\label{prop: generalita su metrica Riemanniana holo}
	Let $g=g(c_1,\overline{c_2})$ be a Bers metric.
	\begin{enumerate}
		\item $\inner{q_1, q'_1}_g=0$, for all $q_1, q'_1\in \HQD(c_1)$.
		\item $\inner{\overline{q_2}, \overline{q_2}'}_g=0$ for all $\overline {q_2}, \overline{q_2}'\in \HQD(\overline{c_2})$\ .
		\item Let $g=\varrho dzd\overline w$, $q_1=\phee dz^2 \in \HQD(c_1)$, $\overline{q_2}=\overline{\psi} d\overline w^2\in \HQD(\overline{c_2})$, then
		\[
		\inner{q_1, \overline{q_2}}_g=\frac 1 4  i \int_S \frac{\phee\cdot \overline{\psi}}{\varrho}dz\wedge d\overline w \ .
		\]
		\item if $g'=g(\phi_1^*(c_1), \phi_2^*(\overline{c_2} ))$, with $\phi_1,\phi_2\in \mathrm{Diff}_0(S)$, then 
\[
\inner{q_1, \overline{q_2}}_g = \inner{\phi_1^*(q_1),\phi_2^*( \overline{q_2})}_{g'}\ .
\]

As a result, the bilinear form $\inners$ defined in Equation \eqref{eq: def metrica in generale} induces a $\mathbb C$-bilinear form on $T \QF(S)$.

\item  The restriction of the bilinear form $\inners$ to the Fuchsian locus is real-valued and coincides with the Weil-Petersson metric.

	\end{enumerate}

\end{prop}
\begin{proof}
	\begin{enumerate}
		\item Consider the frame $X_1:=\partial_{\overline z}$ and $X_2:=\partial_w$, we have that $q_1(X_i, X_j)\ne 0$ if and only if $i=j=2$, but $g(\partial_{\overline z},\partial_{\overline z})=0$ hence $g^{22}\equiv0$: by Equation \eqref{eq: metric on symmetric forms} we conclude that $<q_1,q_1'>_g=0$.
		\item As in the previous step.
		\item Follows directly by Lemma \ref{lemma: metric on hqd} and Equation \eqref{eq: def metrica in generale}.
	\item This follows from the previous step and from Lemma \ref{lemma: isotopy invariance of the integral}.
	
	\item 
	Recall the description of the differential of the Fuchsian map $\mathcal T(S)\xrightarrow{\sim} \mathrm{Fuch}(S)\subset \QF(S)$ in Remark \ref{rmk: differential of the fuchsian map}. We therefore see that the pushed-forward Weil-Petersson metric on $\mathrm{Fuch}(S)$ (defined as in Section \ref{introsection: WP metric}) is given by
	\[
	\inner{q+\overline q, q+\overline q}_{WP}= \left( \int_S \frac{\phee \overline \phee}{\varrho_0} \cdot  \frac i 2 dz\wedge d\overline z \right) \ .
	\]
	Using the previous steps, for all hyperbolic metric $g_0=g(c,c)$ on $S$ we have $\inner{q+\overline q, q+\overline q}_{WP} =   2\inner{q,\overline q}_{g_0} = \inner{q+\overline q, q+\overline q}_{g_0}$.
	\end{enumerate}
\end{proof}

\subsection{The bilinear form is holomorphic Riemannian}
\label{section: the bilinear form is holo Riem}

The aim of this section is to prove that $\inners$ is in fact a holomorphic Riemannian metric.

We fix a framework that will be useful for further calculations.
Fix $(c_1^0, \overline{c_2}^0)\in \CSCS$.

Let $q_1\in \HQD([c_1^0])$. In the identification $\QF([c^0_1], \bullet)\cong\mathcal T(\overline S)$, let $\alpha_{q_1}\in \Gamma(\mathcal T(\overline S))$ be the holomorphic vector field corresponding to infinitesimal deformations by $q_1$, namely $${\alpha_{q_1}}|_{[\overline{c_2}]}=\frac d {dt}_{|t=0} (g(c_1^0, \overline{c_2}) +tq_1) .$$
In the identification $\QF(S)\cong\mathcal T(S)\times\mathcal T(\overline S)$, we also define the holomorphic vector field $L_{q_1}\in \Gamma(T\QF(S))$ as 
\[
L_{q_1}\cong(0, \alpha_{q_1})
\]

Similarly, for all $\overline{q_2}\in \HQD([\overline{c_2}^0])$, we define $\beta_{\overline{q_2}}\in \Gamma (\mathcal T (S))$ as the vector field corresponding to $\overline{q_2}$ in the identification $\mathcal T(\overline S)\cong \QF(\bullet, [\overline{c_2}^0])$ and $R_{\overline{q_2}}\in \Gamma(T\QF(S))$ as the holomorphic vector field corresponding to $({\beta}_{\overline{q_2}},0)$ in the splitting $T\QF(S)\cong T\mathcal T(S)\oplus T\mathcal T(\overline S)$.

\begin{lemma}
	\label{lem: good frame}
	With the above notation:
	\begin{enumerate}[noitemsep,topsep=0pt]
    \item $L_{q_1}$ coincides with $q_1$ on $\QF([c_1^0], \bullet)$ and $R_{\overline{q_2}}$ coincides with $\overline{q_2}$ on $\QF(\bullet, [\overline{c_2}])$.
    	\item $\inner{L_{q_1}, L_{q_1'}}\equiv \inner{R_{\overline{q_2}}, R_{\overline{q_2}'}}\equiv 0$.
		\item $[ L_{ q_1}, R_{\overline {q_2}}]\equiv[R_{\overline{q_2}}, R_{\overline{q_2}'}]\equiv[L_{q_1}, L_{q_1'}]\equiv 0$.
	\end{enumerate}
\end{lemma}
\begin{proof}
Items $(1)$ and $(2)$ follow directly by construction and by the fact that Bers slices are isotropic.
To prove item $(3)$, observe that the vector fields $q_1$ and $q'_1$ commute on $\QF([c^0_1],\bullet)$, as a consequence of the fact that, via the map ${Schw}_+\colon\QF([c^0_1],\bullet)\hookrightarrow\HQD([c_1^0])$, they correspond to the infinitesimal sum by a vector in a vector space. As a result, $\alpha_{q_1}$ and $\alpha_{{q'_1}}$ commute on $\mathcal T(\overline S)$. By applying the analog argument for $R$-type vector fields, the statement then follows by seeing the Lie brackets in the decomposition $T\QF(S)\cong T\mathcal T(S)\oplus T\mathcal T(\overline S)$.
\end{proof}

\begin{lemma}
	\label{lemma: inners tra Lk e Rj costante sulla slice}
	Let $q_1\in \HQD(c_1^0)$ and $\overline{q_2}\in \HQD(\overline{c_2}^0)$. 
	
	With the above notation, $\inner{L_{q_1}, R_{\overline{q_2}} }$ is constant on $\QF([c_1^0],\bullet)\cup \QF(\bullet, [\overline{c_2}^0])$, and it coincides with $\frac 1 2 q_1(\beta_{\overline{q_2}})= \frac 1 2\overline{q_2}(\overline{\alpha}_{q_1}) $. 
\end{lemma}

\begin{proof}
	Denote $L=L_{q_1}$ and $R=R_{\overline{q_2}}$.
	Let $\overline{c_2}\in \mathcal C(\overline S)$, let $g=g(c_1^0, \overline{c_2})=\varrho dz_0 d\overline w$, where $z_0$ is a complex coordinate for $c_1^0$ and $\overline w$ for $\overline{c_2}$. 
	
	On $\QF([c_1^0], \bullet)$, $L$ coincides with $q_1=:\phee dz_0^2$. On the other hand, $R$ in $([c_1^0], [\overline{c_2}])$ is tangent to $\QF(\bullet, [\overline{c_2}])$, so in the metric model \eqref{eq: tangent model} for $\QF(S)$ it coincides in $([c_1^0], [\overline{c_2}])$ with some holomorphic quadratic differential: we denote \[R_{| ([c_1^0], [\overline{c_2}])}=: \overline \tau={\overline \eta d\overline w^2}\in \HQD(\overline{c_2})<T_{([c_1^0], [\overline{c_2}])}\QF(S)\ .\] 
	
	Recall that, on the Bers slice $\QF(\bullet, [\overline{c_2}])\cong \mathcal T(S)$, the vector field $R$ coincides with the vector field ${\beta_{\overline{q_2}}}\in \Gamma (\mathcal T(S))=\Gamma (T\mathrm{Belt}( S ) )$: under the identification in Remark \ref{rmk: models for the tangent}, we get that ${\beta_{\overline{q_2}} }_{|c_1^0}= \left[ \frac{{\overline \eta  } \partial_{\overline{z_0}}\overline w}{\varrho} \frac{d\overline {z_0}}{dz_0} \right]\in \mathrm{Belt}({c_1^0})$. Finally, 
		\begin{align*}
			\inner{L, R} _{|([c_1^0], [\overline{c_2}])} =& \inner{q_1,{\overline{\tau}} }_g= \frac 1 4 i \int_S \frac{\phee {\overline \eta}}{\varrho} dz_0\wedge d\overline w =\\
			&=  \frac 1 2  \int_S \frac{\phee {\overline \eta} \partial_{\overline {z_0} }\overline w} {\varrho} \frac i 2 dz_0\wedge d\overline{z_0} =\frac 1 2 q_1(\beta_{\overline{q_2}})
		\end{align*}
	for all $\overline{c_2}\in \mathcal C(\overline S)$.
	
		In the same fashion, one gets that $\inner{L, R}$ coincides with $\frac 1 2\overline{q_2}(\overline{\alpha}_{q_1}) $ on every point of $\QF(\bullet, [\overline{c_2}])$.
		As a result, one gets that $\frac 1 2 q_1(\beta_{\overline{q_2}})= \frac 1 2\overline{q_2}(\overline{\alpha}_{q_1}) $ which is the value of $	\inner{L, R}$ in $([c_1^0], [\overline{c_2}^0])$.
\end{proof}
Observe that the fact that $\frac 1 2 q_1(\beta_{\overline{q_2}})= \frac 1 2\overline{q_2}(\overline{\alpha}_{q_1}) $ together with Theorem \ref{thm: schwartzian} lead to an alternative proof of McMullen's quasi-Fuchsian reciprocity (see Proposition \ref{prop: new McMullen qF}).

\begin{theorem}
	\label{thm: holomorphic Riemannian}
	The bilinear form $\inners$ on $\QF(S)$ is a holomorphic Riemannian metric.
\end{theorem}
\begin{proof}
	We have to prove that $\inners$ is non-degenerate and holomorphic. 
	
	The fact that the metric is non-degenerate follows directly from Lemma \ref{lemma: inners tra Lk e Rj costante sulla slice}: for all $([c_1^0], [\overline{c_2}^0])\in \TSTS$ one can evaluate the metric on each pair of vector fields $L_{q_1}$ and $R_{\overline{q_2}}$ as defined above, then the pairing on $\mathrm{Belt}{([c])}\times \HQD([c])$ being non-degenerate (see Section \ref{introsection T and T*}) implies that the bilinear form on $\QF(S)$ is non-degenerate as well.
	
	To prove holomorphicity, let $X,Y$ be two holomorphic vector fields on $\QF(S)$. Then the restriction of $\inner{X,Y}$ on each Bers slice $\QF([c^0_1], \bullet)$ is holomorphic, because it can be written as a linear combination through holomorphic functions of the terms $\inner{L_{q_1}, R_{\overline{q_2}}}$, which are constant on the slice $\QF([c^0_1], \bullet)$. Similarly $\inner{X,Y}$ is holomorphic on each left Bers slice $\QF(\bullet, [\overline{c_2}^0])$. Finally, $\inner{X,Y}$ is holomorphic on $\QF(S)$ since its derivative with respect to any antiholomorphic vector is zero.
\end{proof}

\subsection{Uniqueness of holomorphic extensions}

We prove that there is a unique extension of the Weil-Petersson metric to the quasi-Fuchsian space. This is a simple consequence of the complex-analiticity of the metric, but we show it explicitly. 

Let 
\begin{equation}
	\label{eq: holo map cocycles}
	\Lambda\colon \QF(S)\to \Omega \subset \mathbb C^{6\mathrm g -6}
\end{equation}

be a biholomorphism such that the image of the Fuchsian locus is $\Omega\cap \mathbb R^{6\mathrm g -6}$. There are several ways to choose $\Lambda$: for instance, fixing a geodesic lamination $\lambda$ for $S$, $QF(S)$ can be parametrized with an open subset of the vector space of $\faktor{\mathbb C}{2\pi i\mathbb Z}$-valued cocycles along $\lambda$ through the so-called shear-bend coordinates, with $\mathbb R$-valued cocycles corresponding to Fuchsian representations: see \cite{shearbend} for further details.

\begin{lemma}
	\label{lemma: extension of zero is zero}
	Let $\Omega\subset \mathbb C^{N}$ be an open connected subset intersecting $\mathbb R^N$. 
	If $f\colon \Omega\to \mathbb C$ is a holomorphic function such that $f_{|\Omega\cap \mathbb R^N}\equiv 0$, then $f\equiv 0$.
\end{lemma}
\begin{proof}
	This is a simple consequence of the fact that $f$ is complex-analytic. 
	For all $x^0\in \Omega\cap \mathbb R^N$, $f$ can be locally written as $f(z)=\sum_{n=0}^\infty \sum_{k=1}^N a_{k,n}(z_k-x^0_k)^n$, and its restriction to $\mathbb R$ being analytic implies that all the coefficients $a_{k,n}$ are zero. So $f\equiv 0$ in a neighborhood of $\Omega\subset \mathbb R$, implies that $f\equiv 0$ on $\Omega$ because $f$ is holomorphic and $\Omega$ is connected.
\end{proof}

\begin{lemma}
	\label{lemma: holo extension of tensor}
	A $(p,q)$-tensor on the Fuchsian locus $\mathrm {Fuch}(S)$ extends to at most one $\C$-bilinear holomorphic $(p,q)$-tensor on $\QF(S)$.
\end{lemma}
\begin{proof}
	Denote by $\inners_{\C^n}$ the standard holomorphic Riemannian metric on $\C^n$ defined by $\inner{v,w}_{\C^n}=\sum_{k=1}^n v_k w_k$. This metric induces a nondegenerate holomorphic symmetric bilinear form (that we will still denote by $\inners_{\C^n}$) on the tensor bundle $\otimes^{p} T\C^n$ defined by $\inner{ v^1\otimes\dots \otimes v^p, w^1\otimes \dots \otimes w^p}_{\C^n}=\inner{v^1,w^1}_{\C^n}\cdot\ \dots\ \cdot\inner{v^p,w^p}_{\C^n}$.
	
		Assume there exist two holomorphic extensions of a $(p,q)$-tensor on $\mathrm {Fuch}(S)$. By pushing them forward through $\Lambda$, we get two holomorphic $(p,q)$-tensors on $\Omega \subset \mathbb C^{6\mathrm g -6}$ which coincide on $\Omega\cap \mathbb R^{6\mathrm g -6}\ne \emptyset$. Denote by $\alpha$ the difference between them, hence $\alpha\equiv 0$ on $\Omega\cap\mathbb R^{6\mathrm g -6}$. Then, for all $j,k=1,\dots, 6\mathrm g -6$, the vector fields $\partial_{x_j}$ and $\partial_{x_k}$ are holomorphic on $\Omega$, and $\inner{\alpha(\partial_{x_{n_1}},\dots, \partial_{x_{n_q}}), \partial_{x_{m_1}}\otimes\dots\otimes \partial_{x_{m_p}}} \equiv 0$ on $\Omega\cap \mathbb R^{6\mathrm g -6}$. By Lemma \ref{lemma: extension of zero is zero}, the holomorphic function  $\inner{\alpha(\partial_{x_{n_1}},\dots, \partial_{x_{n_q}}), \partial_{x_{m_1}}\otimes\dots\otimes \partial_{x_{m_p}}}$ is everywhere zero on $\Omega$ for any choice of the indeces, hence $\alpha\equiv 0$ on $\Omega$ since $\{\partial_{x_j}\}_j$ is a $\mathbb C$-linear basis and $\inners_{\C}$ is non-degenerate.
\end{proof}

\begin{cor}
	\label{cor: the metric is unique}
	There exists a unique holomorphic Riemannian metric on $\QF(S)$ which extends the Weil-Petersson metric on the Fuchsian locus. Hence $\inners$ coincides with the holomorphic Riemannian metric defined in \cite{LoustauSanders}.
\end{cor}

\begin{cor}
	\label{cor: mapping class group invariance}
	The holomorphic Riemannian metric $\inners$ on $\QF(S)$ is invariant under the diagonal action of the mapping class group on $\QF(S)\cong \TSTS$.
\end{cor}
\begin{proof}
	The mapping class group acts on $\mathcal T(S)$ preserving the Weil-Petersson metric and the complex structure. As a consequence, the diagonal action on $\TSTS$  is by biholomorphisms as well and preserves the Weil-Petersson metric on the Fuchsian locus: the action must therefore be by isometries for the unique holomorphic extension of the Weil-Petersson metric to $\QF(S)$.
\end{proof}

\subsection{The Goldman symplectic form}

The relation between the Goldman symplectic form on the character variety of $\PSL(2,\R)$ and the Weil-Petersson metric (see Section \ref{introsection: WP metric}) allows to give an integral description of the analog form for $\PSL(2,\C)$ on the quasi-Fuchsian locus through its relation with the holomorphic Riemannian metric, as stated in the following Proposition.

\begin{prop}
	\label{prop: Goldman symp}
The $\PSL(2,\C)$-Goldman symplectic form $\omega_G$ vanishes on the Bers slices of $\QF(S)$, and in each point $([c_1], [\overline{c_2}])$ it is such that 
\[
\omega_G([q_1], [\overline{q_2}])= 8i \inner{[q_1], [\overline{q_2}]}=2i \int \frac{\phee \cdot \overline{\psi} }\varrho dz\wedge d\overline w\ 
\]
where $z$ and $\overline w$ are local complex coordinates for $c_1$ and $\overline{c_2}$, $q_1=\phee dz^2$, $\overline{q_2}=\overline \psi d\overline w ^2$, $g=\varrho dzd\overline w$.
\end{prop}
\begin{proof}
We underline that, while the integral description is new, the fact that Bers slices are Lagrangian for $\omega_G$ is well known, and also the relation between $\omega_G$ and $\inners$ has been previously discussed in \cite{LoustauSanders}. 
Nevertheless, we give a full new proof of the whole statement.

Using $T_{([c],[\overline c])} \mathrm{Fuch}(S)=\{q+\overline q\ |\ q\in \HQD(c)\}$ (as seen in Remark \ref{rmk: differential of the fuchsian map}) and the desciption of the Goldman symplectic form with respect to the Weil-Petersson metric (see Section \ref{introsection: WP metric} ), we have 
\[
\omega_G(q+\overline q, q'+\overline{q}') = 8\inner{iq+ \overline{iq}, q'+\overline q ' }= 8i\inner{q-\overline q, q'+ \overline q'} \ 
\]
for all $q\in \HQD(c)$.

Through this and the fact that $\omega_G$ is $\C$-bilinear, we get a full description of $\omega_G$ on the holomorphic tangent bundle to $\QF(S)$ along the Fuchsian locus: indeed, for all $q,q'\in \HQD(c)$,
\begin{align*}
	4\omega_G(q, q')=& \omega_G(q+  \overline q - i({iq}+ \overline{iq}), q'+  \overline q' - i({iq}'+ \overline{iq'}))=\\
	=& \omega_G(q+\overline q, q'+\overline q') -i\omega_G(q+\overline q, {iq}'+\overline{iq'}) -\\
	&- i\omega_G(iq+ \overline{iq}, q'+\overline q ') - \omega_G(iq+\overline{iq}, iq'+\overline{iq'})=\\
	=&8\inner{iq+\overline{iq}, q'+\overline q'}-8i\inner{iq+\overline{iq}, iq'+\overline{iq'}}-\\
		&-8i\inner{-q-\overline q, q'+\overline q'} -\inner{-q-\overline q, iq'+\overline {iq' }}= 32i \inner{q, q'}=0
\end{align*}
and in a similar fashion one gets that 
\begin{align*}
	4\omega_G(q, \overline q')=&  \omega_G(q+  \overline q - i(\overline{iq}+ \overline{iq}), q'+  \overline q' + i(\overline{iq'}+ \overline{iq'}))\\
	=&32i\inner{q, \overline q'} \ .
\end{align*}
 By Lemma \ref{lemma: holo extension of tensor}, $\omega_G$ is the unique holomorphic extension of its restriction on the Fuchsian locus. Since $\inners$ is holomorphic, such an extension can be obtained as follows: for all $X_1, X_2$ holomorphic vector fields on $\QF(S)$ tangent to the foliation of right Bers slices $\{\QF(\bullet, [\overline{c_2}])\}_{[\overline{c_2}]\in \mathcal T(\overline S)}$ and $Y_1, Y_2$ holomorphic vector fields on $\QF(S)$ tangent to the foliation of left Bers slices $\{\QF([c_1], \bullet)\}_{[c_1]\in \mathcal T(S)} $, one defines $\omega_G(X_1, Y_1)= 8i \inner{X_1,Y_1}$, $\omega_G(Y_1, X_1)=-8i \inner{X_1, Y_1}$, $\omega_G(X_1, X_2)=\omega_G(Y_1, Y_2)=0$. We therefore get the statement, with the integral description following from Proposition \ref{prop: generalita su metrica Riemanniana holo}.
\end{proof}

\subsection{A formula for the curvature}

We conclude this section with a few computations on the metric $\inners$.

In order to make computations, we will use the same notation as in Section \ref{section: the bilinear form is holo Riem}.

Fix $\ccpair\in \CSCS$. Fix a $\mathbb C$-linear basis $\{q_1^1, \dots q_1^{6\mathrm g -6}\}$ for $\HQD(c_1)$ and $\{\overline{q_2}^1, \dots \overline{q_2}^{6\mathrm g -6}\}$ for $\HQD(\overline{c_2})$. 
and denote 
\[
L_k= L_{q_1^k}\qquad\qquad\qquad R_{j}=R_{\overline{q_2}^j}\ .
\]

We will also use extensively the \emph{Koszul formula}:
\begin{equation}
	\label{eq: Koszul formula}
	2\inner{\nabla_X Y, Z}= \partial_X(\inner{Y, Z})+\partial_Y (\inner{X,Z}) - \partial_Z(\inner{X,Y})-\inner{[Y,X], Z}-\inner{[X,Z],Y} -\inner{[Y,Z],X}
\end{equation}
where $X,Y,Z$ are vector fields on $\QF(S)$ and $\nabla$ is the Levi-Civita connection of $\inners$.

\begin{lemma}
	\label{lemma: technical nablas}
	\begin{itemize}[noitemsep,topsep=0pt] 
		\item  $\nabla_{L_i} L_k=0$ on $\QF([c_1], \bullet)$ for all $i,k=1,\dots, 6\mathrm g-6$.
		\item $\nabla_{R_i} R_j=0$ on $\QF(\bullet, [\overline{c_2}])$ for all $i,j=1, \dots, 6\mathrm g-6$.
		\item The vector field $\nabla_{L_k}R_j= \nabla_{R_j}L_k$ is tangent to the Bers slices $\QF([c_1], \bullet)$ and $\QF(\bullet, [\overline{c_2}])$ and it is zero in $([c_1],[\overline{c_2}])$.
	\end{itemize}
\end{lemma}
\begin{proof}
	All the proofs follow from straightforwardly from the Koszul formula \eqref{eq: Koszul formula}, by Remark \ref{lem: good frame}, and by Lemma \ref{lemma: inners tra Lk e Rj costante sulla slice}, implying that every two vector fields among $L_1, \dots, L_{6\mathrm g-6}$ and $R_1, \dots, R_{6\mathrm g -6}$ commute and are such that their inner product is constant on  $\QF([c_1], \bullet) \cup \QF(\bullet, [\overline{c_2}])$, hence its derivative is zero over each of these two Bers slices. 

	We only prove explicity the last statement. By the Koszul formula in Equation \eqref{eq: Koszul formula} we get that along $\QF([c_1], \bullet)$ 
	\[
	\inner{\nabla_{L_k}R_j, L_i}= \partial_{L_k} \inner{ R_j, L_i} - \partial_{L_i} \inner{R_j, L_k} \equiv 0\ ,
	\]
	so $\nabla_{L_k}R_j$ is tangent to the Bers slice $\QF([c_1], \bullet)$. With the analog argument, one gets that the vector field $\nabla_{R_j}L_k= \nabla_{L_k}R_j$ is also tangent to $\QF(\bullet, [\overline{c_2}])$, so it must be zero in $\ccpairclass$.
\end{proof}

Let $g=g\ccpair$, let $U_g=\{q_1\in \HQD(c_1)\ |\ g+q_1\text{ is a Bers metric}\}$, and $V_g=\{\overline{q_2}\in \HQD(\overline{c_2})\ |\ g+\overline{q_2}\text{ is a Bers metric}\}$.	

Consider on $U_g$ and $V_g$ the trivial affine connection inherited from the fact that $\HQD(c_1)$ and $\HQD(\overline{c_2})$ are finite-dimensional vector spaces, linearly isomorphic to $\mathbb R^{6\mathrm g-6}$. 
\begin{prop}
	\label{prop: param are affine}
	The maps  
	\begin{equation*}
		\begin{split}
		u\colon	U_g &\to \QF([c_1], \bullet)\\
			q_1&\mapsto [g+q_1]
		\end{split} \qquad \qquad 
		\begin{split}
		v\colon	V_g&\to \QF(\bullet, [\overline{c_2}])\\
			\overline{q_2}&\mapsto [g+\overline{q_2}]
		\end{split}\ .
	\end{equation*}
	are affine diffeomorphisms onto their image if you endow $\QF(S)$ with the Levi-Civita connection $\nabla$ of $\inners$.
\end{prop}
\begin{proof}
	
	By construction of the vector fields $R_j$ and $L_k$, the statement is equivalent to proving that $\nabla_{R_k} R_j\equiv0$ on $\QF(\bullet, [\overline{c_2}])$ and $\nabla_{L_k} L_j\equiv 0$ on $\QF([c_1], \bullet)$, which follow from Lemma \ref{lemma: technical nablas}
\end{proof}
\begin{cor}
	With the notations above, the curves $t\mapsto [g+tq_1]$ are geodesics.
\end{cor}

We can now give some explicit computation for the curvature tensor of $\inners$.

We keep using the notation above. Fix $\ccpairclass\in \TSTS$, and consider the global vector fields $R_1, \dots R_{6\mathrm g-6}, L_1, \dots L_{6\mathrm g-6}$.

In the following, $\mathrm R$ denotes the curvature tensor of $\inners$, namely $\RR(X,Y,X,Y)=\inner{\nabla_X\nabla_Y X - \nabla_Y \nabla_X X- \nabla_{[X,Y]} X,Y}$.
\begin{lemma}
	On the whole $\QF(S)$, for each $j,k=1,\dots, 6\mathrm g -6$,
	\[
	\RR(L_k, R_j, L_k, R_j)= -\partial_{L_k}\partial_{R_j} \inner{L_k, R_j} + \inner{\nabla_{L_k} L_k, \nabla_{R_j}R_j} - \inner{\nabla_{L_k}R_j, \nabla_{L_k}R_j}\ .
	\]
	In particular, on $\QF([c_1], \bullet)\cup \QF(\bullet, [\overline{c_2}])$, 
	\[	\RR(L_k, R_j, L_k, R_j)= -\partial_{L_k}\partial_{R_j} \inner{L_k, R_j} \ .
	\]
\end{lemma}
\begin{proof}
	In the following, we denote $X=L_k$ and $Y=R_j$
	\begin{align*}
		\RR(X,Y,X,Y)&=\inner{\nabla_X \nabla_Y X, Y} -\inner{\nabla_Y\nabla_X X, Y} =\\
		&=\partial_X(\inner{\nabla_YX, Y})- \inner{\nabla_Y X, \nabla_XY} -\partial_Y\inner{\nabla_X X, Y}+\inner{\nabla_X X, \nabla_Y Y}=\\
		&=\partial_X(\inner{\nabla_YX, Y})- \inner{\nabla_Y X, \nabla_XY} -\partial_Y\partial_X\inner{X, Y}+ \partial_Y \inner{X, \nabla_XY}+\inner{\nabla_X X, \nabla_Y Y}.
	\end{align*}
	Now, observe that since $\inner{X,X}=\inner{Y,Y}=0$, the Koszul formula shows immediately that $\inner{\nabla_YX, Y}=0= \inner{X, \nabla_XY}$. As a result, we have proved the first equality. 
	
	To prove the second equality, by Lemma \ref{lemma: technical nablas} we have $\nabla_{L_k} L_k\equiv 0$ on $\QF([c_1], \bullet)$ and $\nabla_{R_j} R_j\equiv 0$ on $\QF(\bullet, [\overline{c_2}])$, so  $\inner{\nabla_{L_k} L_k, \nabla_{R_j}R_j}=0$ on $\QF([c_1], \bullet)\cup \QF(\bullet, [\overline{c_2}])$. Finally, by the last statement of Lemma \ref{lemma: technical nablas}, $\nabla_{L_k}R_j$ is tangent to $\QF([c_1], \bullet)$ and $\QF(\bullet, [\overline{c_2}])$, so it is isotropic.
\end{proof}

\begin{prop}
	The sectional curvature $\mathrm K(L_k, R_j)$, defined when $\inner{L_k,R_j}\ne 0$, in  $\QF([c_1], \bullet)\cup \QF(\bullet, [\overline{c_2}])$ is given by
	\[
	\mathrm K (L_k, R_j)= - \frac{\partial_{L_k}\partial_{R_j} \inner{L_k, R_j} }{\inner{L_k, R_j} ^2}= \partial_{L_k}\partial_{R_j} \log(\inner{L_k, R_j})
	\]
	
\end{prop}

\section{Bounds for the Schwarzian of Bers projective structures}

Let $\ccpairclass\in \TSTS$. Denote
\begin{align*}
	&\mathcal S^+_{[c_1]}:= Schw_+\Big(\QF([c_1], \bullet)\Big)\subset \HQD([c_1]) \qquad  \text{and} \\ 
	&\mathcal S^-_{[\overline{c_2}]}:=  Schw_-\Big( \QF(\bullet, [\overline{c_2}]) \Big)\subset \HQD([\overline{c_2}])\ .
\end{align*}

One of the key observations of this paper (in Proposition \ref{prop: star-shaped}) is that, given a Bers metric $g=g(c_1, \overline{c_2})$, for all $q_1\in \HQD(c_1)$ \emph{small enough} $g+q_1$ is a Bers metric. Theorem \ref{thm: schwartzian} motivates the question of understanding better what "\emph{small enough}" means: if $g+q_1$ is a Bers metric, then $Schw_+([g])-\frac 1 2 [q_1]\in \mathcal S^+_{[c_1]}$. 

In this section, we elaborate this remark, and use the metric formalism to get a lower bound for the distance of a point in $\mathcal S^+_{[c_1]}$ (resp. $\mathcal S^-_{[\overline{c_2}]}$) from its boundary in $\HQD([c_1])$ (resp. $\HQD([\overline{c_2}])$).

\begin{lemma}
	\label{lemma: bound qF}
	Let $g=\varrho dz d\overline w$ be a positive complex  metric, denote $(c_+,c_-)(g)=(c_1, \overline{c_2})$ and let $q_1=\phee dz^2\in \HQD(c_1)$. 
	
	The following statements are equivalent.
	\begin{enumerate}
		\item $g+q_1$ is a positive complex metric.
		\item $g+q_1$ has no non-zero isotropic vectors in the real tangent space.
		\item $|\varrho \partial_z \overline w +\phee|<|\varrho \partial_{\overline z} \overline w|$
	\end{enumerate}
\end{lemma}
\begin{proof}
	We prove that $(1) \Longleftrightarrow (2)$ and $(2) \Longleftrightarrow(3)$.
\begin{itemize}[noitemsep,topsep=0pt]
	\item[(1) $\Rightarrow$ (2)] trivial.
	\item[(2) $\Leftarrow$ (1)] By Proposition \ref{prop: star-shaped}, we know that $g+q_1$ is a Bers metric if and only if it is a positive complex metric. Assume that $g+q_1$ is not a positive complex metric, and assume by contradiction that $g+q_1$ has no non-zero isotropic vectors in the real tangent bundle. In some point of $S$ the isotropic directions $g+q_1$ lie in the same connected component of $\mathbb P(\mathbb C T S)\setminus \mathbb P( T S)$; on the other hand, in a zero of $q_1$ the isotropic directions coincide with those of $g$, hence they lie in opposite connected components. By a continuity argument, there exists a point where an isotropic direction lies in the complex span of a non-zero vector of the real tangent bundle $TS$: since complex metrics are $\mathbb C$-bilinear, there must be an isotropic non-zero vector in $TS$ for $g+q_1$.
	
	\item[(2) $\Longleftrightarrow$ (3)]
		We have that $g+q_1=(\varrho d\overline w +\phee dz)\cdot dz$. Observe that 
	\begin{align*}
	&	\qquad g+q_1 \text{ satisfies $(2)$} &&\Longleftrightarrow&&\\
	\Longleftrightarrow	\qquad&\exists p\in \widetilde S \quad \exists \alpha\in \mathbb C^* \text{ s.t. in $p$: }\qquad &&(g+q_1)(\alpha \partial_z+ \overline \alpha \partial_{\overline z}, \alpha \partial_z+ \overline \alpha \partial_{\overline z})=0&& \Longleftrightarrow \\
	\Longleftrightarrow \qquad&	\exists p\in \widetilde S \quad \exists \alpha\in \mathbb C^*\  \text{ s.t. in $p$: }\qquad &&(\varrho d\overline w +\phee dz)(\alpha \partial_z+ \overline \alpha \partial_{\overline z})=0 &&\Longleftrightarrow \\
	\Longleftrightarrow \qquad&	\exists p\in \widetilde S \quad \exists \alpha\in \mathbb
		C^*\text{ s.t. in $p$: }\qquad &&\varrho \partial_z\overline w +\phee = -\frac{\overline \alpha}{\alpha}\varrho\partial_{\overline z}\overline w &&\Longleftrightarrow\\				
	\Longleftrightarrow \qquad&	\exists p\in \widetilde S\ \  \text{ s.t. in $p$: } \qquad  && |\varrho\partial_z\overline w +\phee|=|\varrho\partial_{\overline z}\overline w|\ . &&
	\end{align*}
	So, $(2)$ holds if and only if $|\varrho\partial_z\overline w +\phee|\ne|\varrho\partial_{\overline z}\overline w|$ on $\widetilde S$. Since $|\partial_z\overline w|<|\partial_{\overline z} \overline w|$ we have that $|\varrho\partial_z\overline w +\phee|<|\varrho\partial_{\overline z}\overline w|$ in the zeros of $q_1$, so the inequality must hold on the whole surface if and only if $g+q_1$ satisfies $(2)$.
\end{itemize}
	\end{proof}

Before stating the main bound, we introduce some notation. 

For all $\ccpair\in \CSCS$, let $\mu\ccpair$ be the absolute value of the Beltrami differential of $c_2$ with respect to $c_1$, so $$|\mu\ccpair|=\left| \frac{\partial_{\overline z} w}{\partial_z w}\right|$$
where $z$ and $w$ denote local coordinates for $c_1$ and ${c_2}$ respectively. 

Moreover, denoting by $g_0=g(c_1, \overline{c_1})$ the Riemannian hyperbolic metric in the conformal class of $c_1$, we can define the absolute value of $q_1\in \HQD(c_1)$ as the function \[
|q_1|_{g_0} := \frac 12 \left\|{Re(q_1)} \right\|_{g_0}= \left| \frac{q_1}{g_0} \right|= \left|\frac{\phee}{\varrho_0}\right|
\]
where $q_1=\phee dz^2$, $g_0=\varrho_0dzd\overline z$. We can also define the $L^{\infty}$-norm of $q_1$ by $\|q_1\|_{g_0}=\max_S |q_1|_{g_0}$

	\begin{theorem}
		\label{thm: bound qF}
		Let $g=g(c_1,\overline{c_2})\in \CSCS$ be a Bers metric, let $g_0=g(c_1, \overline{c_1})$ be the Riemannian hyperbolic metric in the conformal class of $c_1$ and $\mu=|\mu\ccpair|$.

        If $q_1\in \HQD(c_1)$ is such that in every point
        \begin{equation}
        	\label{eq: bound c1}
            |q_1|_{g_0} <\frac 1 2 (1-|\mu|)\left| \frac{dA_g}{dA_{g_0}} \right|,
        \end{equation}
then $Schw_+([c_1], [\overline{c_2}])+ [q_1] \in \mathcal S^+_{[c_1]}$.

	\end{theorem}

		\begin{proof}[Proof of Theorem \ref{thm: bound qF}.]

				Denote $q_1=\phee dz^2\in \HQD(c_1)$, $g=\varrho dz d\overline w$ and $g_0=\varrho_0 dzd\overline w$.
 
                Observing that $dA_g= \frac i 2 \varrho dz\wedge d\overline w= \frac i 2 \varrho dz\wedge (\partial_z \overline w dz+ \partial_{\overline z}\overline w d\overline z )= \frac i 2 \varrho \partial_{\overline z}\overline w dz\wedge d\overline z$, the condition \eqref{eq: bound c1} is equivalent in local coordinates to
                \[
                |\phee|< \frac 12 (1-|\mu|) |\varrho\partial_{\overline z}\overline w|= \frac {|\varrho|} 2 (|\partial_z w|-|\partial_{\overline z} w|) 
                \]
                from which we deduce that 
				\[
				|\varrho \partial_z \overline w -2\phee|\le |\varrho \partial_z \overline w|+2|\phee|< |\varrho \partial_z \overline w|+ |\varrho \partial_z w|-|\varrho \partial_z\overline w|= |\varrho \partial_z w|\ .
				\]
                
                By Lemma \ref{lemma: bound qF}, we conclude that, under the assumption \eqref{eq: bound c1},
				$g-2q_1$ is a positive complex metric, hence it is a Bers metric by Proposition \ref{prop: star-shaped}. 
				
				By Theorem \ref{thm: schwartzian} 
				\[[\mathbf{Schw}_+(g- 2 q_1)]= [\mathbf{Schw}_+(g) + q_1]= Schw_+([c_1], [\overline{c_2}]) +[q_1]\ 
				\]
				is therefore an element of $\mathcal{S}^+_{[c_1]}$.
   
		\end{proof}

An immediate consequence of Theorem \ref{thm: bound qF} is the following.

\begin{cor}
	\label{cor: ball bound for schw}
Let $\ccpair\in \CSCS$,  $g=g\ccpair$, let $g_0=g(c_1, \overline{c_1})$ be the hyperbolic metric uniformizing $c_1$, and denote by $dA_g$ and $dA_{g_0}$ respectively their area forms.

Let 
\begin{equation}
	\label{eq: definition of R}
	R:= \frac 1 2 \min_S\  \left(  \Big(1- \left|\frac{\partial_{\overline z} w}{\partial_z w} \right| \Big) \left| \frac{dA_g }{dA_{g_0}} \right| \ \right)
\end{equation}

where $z$ and $\overline w$ are any local coordinates for $c_1$ and $\overline{c_2}$ respectively. 

Then, \[B_{\infty}\Big(Schw_+([c_1], [\overline{c_2}]), R\Big)\subset \mathcal S^+_{[c_1]}\, \]
where $B_{\infty}$ denotes the ball with respect to the $L^{\infty}$ norm on $\HQD([c_1])$.
\end{cor}

\begin{remark}
    As mentioned in the introduction, Corollary \ref{cor: ball bound for schw} can be seen as a generalization of the well-known Kraus bound: if $c_1=c_2$, the expression in Equation \eqref{eq: definition of R} just becomes $R=\frac 12$.
\end{remark}

\begin{remark}
	\label{rmk: alternative bounds}
				As shown in the proof of Theorem \ref{thm: bound qF}, \eqref{eq: bound c1} can be written in local coordinates in other forms that don't involve $g_0$, such as:
                \begin{align*}
		|q_1(\partial_z,\partial_z)|<& \Big(1- \left|\frac{\partial_{\overline z} w}{\partial_z w} \right| \Big)  |g(\partial_z, \partial_{\overline z})|=\\
        =&\frac {|\varrho|} 2  \big( |\partial_ z w|-|\partial_ {\overline{z}}w| \big) =  \\
		=& |g(\partial_z, \partial_{\overline z})|- \frac 1 2 |g(\partial_z, \partial_z)|
	\end{align*}
where $z$ and $\overline w$ are local coordinates for $c_1$ and $\overline{c_2}$ respectively and $g=\varrho dzd\overline w$.
\end{remark}

\begin{remark}
	Theorem \ref{thm: bound qF} and Corollary \ref{cor: ball bound for schw} clearly have analogs for $Schw_-$ that can be proved in the same fashion.
\end{remark}

\begin{remark}
	\label{rmk: sharpness of R}
     While $Schw_+([c_1],[\overline{c_2}])$ only depends on the isotopy classes of $c_1$ and $\overline{c_2}$, the description of $R$ in Equation \eqref{eq: definition of R} seems to depend on the choice of the representatives $c_1$ and $\overline{c_2}$ in their isotopy classes. It would be interesting to determine, for $c_1$ fixed, the optimal choice of $\overline{c_2}$ in its isotopy class that maximizes $R$. 
\end{remark}

\section{Revisiting classic results with the metric formalism}

\subsection{The Schwarzian map is affine}

\begin{prop}
	The Schwarzian maps \[Schw_+\colon \QF([c_1], \bullet)\to \HQD([c_1])\qquad \quad \text{and} \qquad \quad Schw_-\colon \QF(\bullet, [\overline{c_2}])\to \HQD([\overline{c_2}])\] are affine diffeomorphisms onto their images.
\end{prop}
\begin{proof}
	The Schwarzian map $Schw_+\colon \QF([c_1], \bullet)\to \HQD([c_1])$ can be locally seen around each point $\ccpairclass\in \TSTS$ as the composition of the inverse of the local parametrization $u\colon U_g\to \QF([c_1], \bullet)$ as in Proposition \ref{prop: param are affine} and the immersion
	\begin{align*}
		U_g &\to \HQD([c_1])\\
		q_1&\mapsto [\mathbf{Schw}_+(g+q_1)]
	\end{align*}
	which is affine by Theorem \ref{thm: schwartzian}.
	The thesis follows in the same fashion for $Schw_-$.
\end{proof}

\subsection{Quasi-Fuchsian reciprocity and the differential of the Schwarzian map}

Let $([c_1], [\overline{c_2}])\in \TSTS$.

Recall the Schwarzian maps $Schw_+( [c_1],\bullet)\colon \mathcal T(\overline S)\to \HQD([{c_1}])$ and $ Schw_-(\bullet, [\overline{c_2}])\colon \mathcal T(S)\to \HQD([\overline{c_2}])$. Since the targets are vector spaces, their differentials are mapped in each point to $\HQD([c_1])$ and $\HQD([\overline{c_2}])$ too

A classic result in quasi-Fuchsian geometry is the following.
\begin{theorem*}[McMullen's quasi-Fuchsian reciprocity \cite{McMullen}]
	
	Let $\ccpair\in \CSCS$, $X\in T_{[c_1]} \mathcal T(S)$, $\overline Y\in T_{[\overline{c_2}]} \mathcal T(\overline S)$. Then,
	\[
	\big(	\partial_{\overline Y} \ Schw_+( [c_1],\bullet)\big) (X) =  \big(	\partial_{X}\ Schw_-(\bullet, [\overline{c_2}])\big) (\overline Y)\ 
	\]
	where we used the identifications $\HQD([c_1])\cong T^*_{[c_1]}\mathcal T(S)$ and $\HQD([\overline{c_2}])\cong T^*_{[\overline{c_2}]}\mathcal T(\overline S)$.
\end{theorem*}

McMullen's reciprocity can actually be seen using the holomorphic extension of the Weil-Petersson metric on $\QF(S)$, providing an alternative description of the derivative of the Schwarzian maps and of the holomorphic Riemannian metric itself.

For all $X\in T_{[c_1]}\mathcal T(S)$ and $\overline Y\in T_{[\overline{c_2}]}\mathcal T(\overline S)$, denote 
\begin{align*}
	\mathbf{ X}= (X,0)\in T_{([c_1], [\overline{c_2}])}\Big(\mathcal T(S)\times \mathcal T(\overline S)\Big)&\\
	\mathbf{\overline Y}= (0, \overline Y)\in T_{([c_1], [\overline{c_2}])}\Big(\mathcal T(S)\times \mathcal T(\overline S)\Big)&\ .
\end{align*}

\begin{prop}
	\label{prop: new McMullen qF}
	Let $g=g(c_1, \overline{c_2})$
	\[
	\big(	\partial_{\overline Y} \ Schw_+( [c_1],\bullet)\big) (X) =  - \inner{\mathbf X, \mathbf{\overline Y} }_{[g]} = \big(	\partial_{ X}\ Schw_-(\bullet, [\overline{c_2}])\big) ({\overline Y})\ ,
	\]
	in particular, one gets McMullen's quasi-Fuchsian reciprocity.

\end{prop}
We remark that a similar result is shown inside a proof of Theorem 5.13 in \cite{LoustauSanders}.

\begin{proof}
	With respect to the metric model for $T\QF(S)$ in \eqref{eq: tangent model}, $\mathbf X$ corresponds to \newline${q_X=:\overline \psi d\overline w^2 \in \HQD(\overline{c_2})<T_{([c_1], [\overline{c_2}])}\QF(S)}$, and $\mathbf Y$ corresponds to $q_Y=:\phee dz^2 \in \HQD(c_1)<T_{([c_1], [\overline{c_2}])}\QF(S)$. 
	By Remark \ref{rmk: models for the tangent}, $X$ and $\overline Y$ correspond to the Beltrami differentials \newline $\beta_{ X}=\frac{\overline \psi \partial_{\overline z} \overline w}{\varrho} \frac{d\overline z}{dz}$ and $\beta_{\overline Y}= \frac{\phee \partial_w z}{\varrho}\frac{dw}{d\overline w}$, respectively.

	Recalling Equations \eqref{eq: schw holo} and \eqref{eq: schw antiholo}.
	\begin{align*}
		\partial_{\overline Y} \big( \mathbf{Schw}_+( g)\big) (\beta_X)&= \left( \frac{d}{dt}_{|0} \mathbf{Schw}_+(g+tq_{\overline Y}) \right) (\beta_X)  = -\frac 1 2 q_{\overline Y} (\beta_X)=\\
		&= -\frac 1 2 \int_S \phee\cdot \frac{ \overline{\psi} \partial_{\overline z} \overline w}{\varrho} \frac i 2 dz\wedge d\overline z= -\frac i 4 \int_S  \frac{ \phee \overline{\psi} }{\varrho} dz\wedge d\overline w= -\inner{\mathbf X, \mathbf{\overline{Y}} }_g\ .
	\end{align*}
	In the same fashion, 
	\begin{align*}
		\partial_{X} \big( \mathbf{Schw}_-( g)\big) (\beta_{\overline Y})&= \left( \frac{d}{dt}_{|0} \mathbf{Schw}_-(g+tq_{{X} }) \right) (\beta_{\overline Y})  = -\frac 1 2 q_{ X} (\beta_{\overline Y})= \\
		&=-\frac 1 2   \int_S \overline{\psi}\cdot \frac{  \phee \partial_{w}  z}{\varrho}  \frac i 2 dw  \wedge d{\overline w}
		= -\frac i 4 \int_S  \frac{ \phee \overline{\psi} }{\varrho} dz\wedge d\overline w= -\inner{\mathbf X, \mathbf{\overline{Y}} }_g \ . 
	\end{align*}
	
\end{proof}

As a consequence, we get an alternative description of the imaginary part of $\inners$ in terms of the Hessian of the renormalized volume $\mathcal V_{Ren}:\QF(S)\to \R$ (see \cite{KrasnovSchlenker}).
\begin{cor}
	Let $\mathbf{X}\in T_{[c_1]} \Big(\mathcal T(S)\times \{[\overline{c_2}]\}\Big)$ and $\overline{\mathbf Y}\in T_{[\overline{c_2}]}  \Big(\{[c_1]\}\times \mathcal T(\overline S) \Big)$. Then,
	\[
	Re(\inner{\mathbf X, {\mathbf {\overline Y} } }_{[g]}) = - 4\partial_{\mathbf X} \partial_{\mathbf{\overline Y}} (\mathcal V_{Ren})
	\]
	
	In other words, 
	\[
	\inner{\mathbf X, {\mathbf {\overline Y} } }_{[g]}= - 4\partial_{\mathbf X} \partial_{\mathbf{\overline Y}} (\mathcal V_{Ren})+  4i \partial_{i\mathbf X} \partial_{\mathbf{\overline Y}} (\mathcal V_{Ren})
	\]
\end{cor}
\begin{proof}
	Let use the notation in the proof of Proposition \ref{prop: new McMullen qF}. By the work of Krasnov and Schlenker (\cite{KrasnovSchlenker}), $Re\left(Schw_+([g]) (\beta_X)\right)= 4\partial_X \mathcal V_{Ren}$. As a consequence \[Re\left(\partial_{X} \big( Schw_+( [g])\big) (\beta_{\overline Y})\right)= Re \left(\partial_{\overline Y} \big( Schw_-( [g])\big) (\beta_X)\right)=  4\partial_{\textbf X} \partial_{\overline {\textbf Y} } (\mathcal V_{Ren}),\] and the result follows by Proposition \ref{prop: new McMullen qF}.
\end{proof}

\printbibliography

@article{LoustauSanders,
title={Bi-Lagrangian structures and Teichm$\backslash$" uller theory},
author={Loustau, Brice and Sanders, Andrew},
journal={arXiv preprint arXiv:1708.09145},
year={2017}
}

@article{dumasprojective,
	title={Complex projective structures},
	author={Dumas, David},
	journal={Handbook of Teichm{\"u}ller theory},
	volume={2},
	pages={455--508},
	year={2009}
}

@incollection {Otalhandbook,
	AUTHOR = {Otal, Jean-Pierre},
	TITLE = {About the embedding of {T}eichm\"{u}ller space in the space of
	geodesic {H}\"{o}lder distributions},
	BOOKTITLE = {Handbook of {T}eichm\"{u}ller theory. {V}ol. {I}},
	SERIES = {IRMA Lect. Math. Theor. Phys.},
	VOLUME = {11},
	PAGES = {223--248},
	PUBLISHER = {Eur. Math. Soc., Z\"{u}rich},
	YEAR = {2007},
	MRCLASS = {32G15 (30F10 57M50 57N05)},
	MRNUMBER = {2349671},
	MRREVIEWER = {Alexander Vasil\cprime ev},
	DOI = {10.4171/029-1/5},
	URL = {https://doi.org/10.4171/029-1/5},
}

@article {BersUniformization,
	AUTHOR = {Bers, Lipman},
	TITLE = {Simultaneous uniformization},
	JOURNAL = {Bull. Amer. Math. Soc.},
	FJOURNAL = {Bulletin of the American Mathematical Society},
	VOLUME = {66},
	YEAR = {1960},
	PAGES = {94--97},
	ISSN = {0002-9904},
	MRCLASS = {30.00},
	MRNUMBER = {111834},
	MRREVIEWER = {H. L. Royden},
	DOI = {10.1090/S0002-9904-1960-10413-2},
	URL = {https://doi.org/10.1090/S0002-9904-1960-10413-2},
}

@article {McMullen,
	AUTHOR = {McMullen, Curtis T.},
	TITLE = {The moduli space of {R}iemann surfaces is {K}\"{a}hler hyperbolic},
	JOURNAL = {Ann. of Math. (2)},
	FJOURNAL = {Annals of Mathematics. Second Series},
	VOLUME = {151},
	YEAR = {2000},
	NUMBER = {1},
	PAGES = {327--357},
	ISSN = {0003-486X},
	MRCLASS = {32G15 (30F60 32Q15 32Q45)},
	MRNUMBER = {1745010},
	MRREVIEWER = {Yoichi Imayoshi},
	DOI = {10.2307/121120},
	URL = {https://doi.org/10.2307/121120},
}

@article {MeBonsante,
	AUTHOR = {Bonsante, Francesco and El Emam, Christian},
	TITLE = {On immersions of surfaces into {$SL(2,\Bbb C)$} and geometric
	consequences},
	JOURNAL = {Int. Math. Res. Not. IMRN},
	FJOURNAL = {International Mathematics Research Notices. IMRN},
	YEAR = {2022},
	NUMBER = {12},
	PAGES = {8803--8864},
	ISSN = {1073-7928},
	MRCLASS = {53C42 (32Q30 53C21 57K31)},
	MRNUMBER = {4436196},
	DOI = {10.1093/imrn/rnab189},
	URL = {https://doi.org/10.1093/imrn/rnab189},
}

@inproceedings{kraus1932zusammenhang,
	title={{\"U}ber den Zusammenhang einiger Charakteristiken eines einfach zusammenh{\"a}ngenden Bereiches mit der Kreisabbildung},
	author={Kraus, Wilhelm},
	year={1932},
	organization={Math. Seminar}
}

@article {Nehari,
	AUTHOR = {Nehari, Zeev},
	TITLE = {The {S}chwarzian derivative and schlicht functions},
	JOURNAL = {Bull. Amer. Math. Soc.},
	FJOURNAL = {Bulletin of the American Mathematical Society},
	VOLUME = {55},
	YEAR = {1949},
	PAGES = {545--551},
	ISSN = {0002-9904},
	MRCLASS = {30.0X},
	MRNUMBER = {29999},
	MRREVIEWER = {M. S. Robertson},
	DOI = {10.1090/S0002-9904-1949-09241-8},
	URL = {https://doi.org/10.1090/S0002-9904-1949-09241-8},
}

@article {Markovic,
	AUTHOR = {Markovic, Vladimir},
	TITLE = {Carath\'{e}odory's metrics on {T}eichm\"{u}ller spaces and
	{$L$}-shaped pillowcases},
	JOURNAL = {Duke Math. J.},
	FJOURNAL = {Duke Mathematical Journal},
	VOLUME = {167},
	YEAR = {2018},
	NUMBER = {3},
	PAGES = {497--535},
	ISSN = {0012-7094},
	MRCLASS = {32G15 (20H10 30F45 30F60 37A17)},
	MRNUMBER = {3761105},
	MRREVIEWER = {Athanase Papadopoulos},
	DOI = {10.1215/00127094-2017-0041},
	URL = {https://doi.org/10.1215/00127094-2017-0041},
}

@article {Lempert,
	AUTHOR = {Lempert, L\'{a}szl\'{o}},
	TITLE = {La m\'{e}trique de {K}obayashi et la repr\'{e}sentation des domaines
	sur la boule},
	JOURNAL = {Bull. Soc. Math. France},
	FJOURNAL = {Bulletin de la Soci\'{e}t\'{e} Math\'{e}matique de France},
	VOLUME = {109},
	YEAR = {1981},
	NUMBER = {4},
	PAGES = {427--474},
	ISSN = {0037-9484},
	MRCLASS = {32H15},
	MRNUMBER = {660145},
	MRREVIEWER = {M. Skwarczy\'{n}ski},
	URL = {http://www.numdam.org/item?id=BSMF_1981__109__427_0},
}

@article {FischerTromba,
	AUTHOR = {Fischer, A. E. and Tromba, A. J.},
	TITLE = {On the {W}eil-{P}etersson metric on {T}eichm\"{u}ller space},
	JOURNAL = {Trans. Amer. Math. Soc.},
	FJOURNAL = {Transactions of the American Mathematical Society},
	VOLUME = {284},
	YEAR = {1984},
	NUMBER = {1},
	PAGES = {319--335},
	ISSN = {0002-9947},
	MRCLASS = {32G15 (58B20)},
	MRNUMBER = {742427},
	MRREVIEWER = {C. Earle},
	DOI = {10.2307/1999289},
	URL = {https://doi.org/10.2307/1999289},
}

@article {BergerEbin,
	AUTHOR = {Berger, M. and Ebin, D.},
	TITLE = {Some decompositions of the space of symmetric tensors on a
	{R}iemannian manifold},
	JOURNAL = {J. Differential Geometry},
	FJOURNAL = {Journal of Differential Geometry},
	VOLUME = {3},
	YEAR = {1969},
	PAGES = {379--392},
	ISSN = {0022-040X},
	MRCLASS = {53.45 (57.00)},
	MRNUMBER = {266084},
	MRREVIEWER = {Y. Katsurada},
	URL = {http://projecteuclid.org/euclid.jdg/1214429060},
}

@article {GoldmanSymp,
	AUTHOR = {Goldman, William M.},
	TITLE = {The symplectic nature of fundamental groups of surfaces},
	JOURNAL = {Adv. in Math.},
	FJOURNAL = {Advances in Mathematics},
	VOLUME = {54},
	YEAR = {1984},
	NUMBER = {2},
	PAGES = {200--225},
	ISSN = {0001-8708},
	MRCLASS = {32G15 (57M05)},
	MRNUMBER = {762512},
	MRREVIEWER = {C. Earle},
	DOI = {10.1016/0001-8708(84)90040-9},
	URL = {https://doi.org/10.1016/0001-8708(84)90040-9},
}

@article{KrasnovSchlenker,
	title={On the renormalized volume of hyperbolic 3-manifolds},
	author={Krasnov, Kirill and Schlenker, Jean-Marc},
	journal={Communications in mathematical physics},
	volume={279},
	pages={637--668},
	year={2008},
	publisher={Springer}
}

@article {Morrey,
	AUTHOR = {Morrey, Jr., Charles B.},
	TITLE = {On the solutions of quasi-linear elliptic partial differential
	equations},
	JOURNAL = {Trans. Amer. Math. Soc.},
	FJOURNAL = {Transactions of the American Mathematical Society},
	VOLUME = {43},
	YEAR = {1938},
	NUMBER = {1},
	PAGES = {126--166},
	ISSN = {0002-9947},
	MRCLASS = {35J60},
	MRNUMBER = {1501936},
	DOI = {10.2307/1989904},
	URL = {https://doi.org/10.2307/1989904},
}

@article {shearbend,
	AUTHOR = {Bonahon, Francis},
	TITLE = {Shearing hyperbolic surfaces, bending pleated surfaces and
	{T}hurston's symplectic form},
	JOURNAL = {Ann. Fac. Sci. Toulouse Math. (6)},
	FJOURNAL = {Toulouse. Facult\'{e} des Sciences. Annales. Math\'{e}matiques. S\'{e}rie
	6},
	VOLUME = {5},
	YEAR = {1996},
	NUMBER = {2},
	PAGES = {233--297},
	ISSN = {0240-2955},
	MRCLASS = {57M50 (53C15 57N05 57N10)},
	MRNUMBER = {1413855},
	MRREVIEWER = {Athanase Papadopoulos},
	URL = {http://www.numdam.org/item?id=AFST_1996_6_5_2_233_0},
}

@article{bridgeman2010hausdorff,
	title={Hausdorff dimension and the Weil--Petersson extension to quasifuchsian space},
	author={Bridgeman, Martin},
	journal={Geometry \& Topology},
	volume={14},
	number={2},
	pages={799--831},
	year={2010},
	publisher={Mathematical Sciences Publishers}
}

@article{bridgeman2015pressure,
	title={The pressure metric for Anosov representations},
	author={Bridgeman, Martin and Canary, Richard and Labourie, Fran{\c{c}}ois and Sambarino, Andres},
	journal={Geometric and Functional Analysis},
	volume={25},
	number={4},
	pages={1089--1179},
	year={2015},
	publisher={Springer}
}

@article{petri2019teichmuller,
	title={Teichm{\"u}ller Theory},
	author={Petri, Bram},
	journal={Lecture notes, Sorbonne University, July},
	year={2019}
}

@book{ahlfors2006lectures,
	title={Lectures on quasiconformal mappings},
	author={Ahlfors, Lars Valerian},
	volume={38},
	year={2006},
	publisher={American Mathematical Soc.}
}

@book{farb2011primer,
  title={A primer on mapping class groups},
  author={Farb, Benson and Margalit, Dan},
  volume={49},
  year={2011},
  publisher={Princeton university press}
}

@book{hubbard2016teichmuller,
	title={Teichm{\"u}ller theory and applications to geometry, topology, and dynamics},
	author={Hubbard, John H},
	volume={2},
	year={2016},
	publisher={Matrix Editions}
}

@misc{ElSa,
      title={On a Bers theorem for SL(3,C)}, 
      author={Christian El Emam and Nathaniel Sagman},
      year={2024},
      eprint={2406.15287},
      archivePrefix={arXiv},
      primaryClass={math.DG},
      url={https://arxiv.org/abs/2406.15287}, 
}

@misc{ESholodependence,
      title={Holomorphic dependence for the Beltrami equation in Sobolev spaces}, 
      author={Christian El Emam and Nathaniel Sagman},
      year={2024},
      eprint={2410.06175},
      archivePrefix={arXiv},
      primaryClass={math.CV},
      url={https://arxiv.org/abs/2410.06175}, 
}

@misc{ElSa2,
      title={Complex harmonic maps and rank 2 higher Teichm\"uller theory}, 
      author={Christian El Emam and Nathaniel Sagman},
      year={2025},
      eprint={2506.11746},
      archivePrefix={arXiv},
      primaryClass={math.DG},
      url={https://arxiv.org/abs/2506.11746}, 
}

\end{document}